\documentclass[10pt,a4paper]{article}

\usepackage{amsfonts}
\usepackage{epsfig}
\usepackage{graphicx}
\usepackage{amsmath}
\usepackage{amssymb}

\def\NN{\mathbb N}
\def\ZZ{\mathbb Z}

\def\vep{\varepsilon}

\def\tr{\hbox{tr}}

\def\ms{\medskip}

\def\beq{\begin{equation}}
\def\eeq{\end{equation}}

\newtheorem{theo}{\bf Theorem}[section]
\newtheorem{prop}{\bf Proposition}[section]
\newtheorem{defi}{\bf Definition}[section]
\newtheorem{cor}{\bf Corollary}[section]
\newtheorem{exe}{\bf Example}[section]
\newtheorem{lem}{\bf Lemma}[section]

\newcommand{\blanksquare}{\,\,\,$\sqcup\!\!\!\!\sqcap$}
\newenvironment{proof}{{\flushleft{\bf Proof: }}}{\hspace{\stretch{1}} \blanksquare\\}

\begin{document}

\title{\textbf{Periodic points of Ruelle-expanding maps}}

\author{Maria Carvalho\thanks{Partially supported by FCT through CMUP } $\,$ and M\'ario Alexandre Magalh\~aes\thanks{Supported by FCT through the PhD grant SFRH/BD/33092/2007} }

\maketitle

\begin{abstract}
We prove that, for a Ruelle-expanding map, the zeta function is rational and the topological entropy is equal to the exponential growth rate of the periodic points.
\end{abstract}

\emph{keywords:} Ruelle-expanding; Zeta function; Entropy.\\

\tableofcontents
\normalsize

\newpage

\section{Introduction}

Given a dynamical system with a finite number of periodic points with period $n$, for each $n\in\NN$, the (Artin-Mazur, \cite{AM}) Riemann \emph{zeta function} is a complex series that encodes all the information regarding the number of these points. More precisely, for a map $f$ with a finite number $N_n(f)$ of periodic points with period $n$, its zeta function is the formal series
\[z \in \mathbb{C} \mapsto \zeta_f(z)=\exp\left(\sum_{n=1}^\infty\frac{N_n(f)}{n}\, z^n\right).\]
If this map admits a meromorphic continuation to the whole complex plane, the poles, zeros and residues of the extended $\zeta$-function provide additional topological invariants for $f$ and an insight into the orbit structure.

It is known that $\zeta_f$ is a rational function when $f$ is a Markov subshift of finite type (unilateral or bilateral) or a $\mathcal{C}^1$ diffeomorphism on a hyperbolic set with local product structure \cite{Bo}. In this work we study another class of maps first introduced in a differentiable setting by M. Shub in \cite{Shu3} and then studied by D. Ruelle in \cite{Ru1} and \cite{Ru2}. Ruelle proposed a more general definition based on a simple metric property: a \emph{Ruelle-expanding} map is just an open continuous transformation, defined on a compact metric space, which expands distances locally (see Section~\ref{RE}). This concept includes Markov unilateral subshifts of finite type and generalizes the notion of $\mathcal{C}^1$ expanding map \cite{Shu3}, defined on manifolds, freeing its essence from the differentiability constraints. Our first result concerning this family of transformations is the following.

\begin{theo}\label{MainTheorem1}
If $f:K \rightarrow K$ is Ruelle-expanding, then $\zeta_f$ is a rational function.
\end{theo}

The proof, in Section~\ref{Rational}, relies on the existence of finite covers of $K$ with arbitrarily small diameter and exhibiting properties that resemble the Markov partitions used to prove the analogous result in the hyperbolic setting \cite{M}. In particular, we will establish a Shadowing Lemma, that enable us to detect periodic orbits, and construct a semiconjugacy between $f$ and an adequate Markov subshift of finite type that will suggest how to count the periodic points of $f$ with given period.

When $f$ is either a C-dense Axiom A diffeomorphism \cite{Bo2}, a piecewise monotone mapping of an interval with positive entropy \cite{MS} or a Markov subshift of finite type determined by an irreducible matrix \cite{Wa}, the topological entropy, say $h(f)$, is given by
$$h(f)=\lim_{n \rightarrow +\infty}\,\,\frac{1}{n}\, \log N_n(f)$$
and so
$$h(f)=-\log\rho$$
where $\rho$ is the radius of convergence of $\zeta_f$. We will also generalize this equality to the Ruelle-expanding setting. It is known \cite{Ru1} that, for a Ruelle-expanding $f$ defined on a compact metric space $(K,d)$, there is a (unique) finite family of compact disjoint subsets (called basic components)
$$\displaystyle \left(\Lambda_i^{(m)}\right)_{i \in \{1, \ldots, n_m\}; \, \, m \in \{1, \ldots, M\}}$$
such that

\begin{itemize}
\item[(C1)] $f(\Lambda_i^{(m)})=\Lambda_{i+1}^{(m)}$ for all $i \in \{1, \ldots, n_m-1\}$ and $m \in \{1, \ldots, M\}$.
\item[(C2)] $f(\Lambda_{n_m}^{(m)})=\Lambda_1^{(m)}$ for all $m \in \{1, \ldots, M\}$.
\item[(C3)] $\bigcup_{i,m}\,\,\Lambda_i^{(m)}=\overline{Per(f)}$.
\item[(C4)] $f^{n_m}|_{\Lambda_i^{(m)}}$ is Ruelle-expanding.
\item[(C5)] For any open nonempty subset $V$ of $\Lambda_i^{(m)}$ there is $N\in \mathbb{N}$ such that $(f^{n_m})^N(V)=\Lambda_i^{(m)}$.
\end{itemize}

\noindent For instance, if $K$ is connected, then $K=\overline{Per(f)}$ and it is equal to one of the basic components, where $f$ is topologically mixing. The topological entropy of the restriction of $f$ to each basic $m$-cycle, say $\left(\Lambda_i^{(m)}\right)_{i \in \{1, \ldots, n_m\}}$, is equal to $\frac{1}{n_m} \, h(f^{n_m})$. Therefore, to relate the entropy with the growth of the periodic points, it is enough to prove that:

\begin{theo}\label{MainTheorem2}
If $f:K \rightarrow K$ is Ruelle-expanding and $K$ is a basic component, then
$$h(f)=\lim_{n \rightarrow +\infty}\,\,\frac{1}{n}\, \log N_n(f).$$
\end{theo}

As we will see, Ruelle-expanding maps are expansive, and so the explicit computation of the topological entropy is possible using either a generator with small enough diameter or separated sets determined by an expansivity constant of $f$ \cite{Wa}. The mixing property of $f$, assisted by a Shadowing Lemma, will provide a method to detect periodic orbits and compare its number, for large enough periods, with the cardinal of maximal separated sets. In this way, the proof of Theorem~\ref{MainTheorem2}, in Section~\ref{Entropy}, will conclude that $h(f)=-\log\rho$ and that $\lim_{n \rightarrow +\infty}\,\,\frac{1}{n}\, \log N_n(f)$ exists.

\section{Basic definitions}

\subsection{Shift}

Let $k$ be a natural number and $[k]$ denote the set $\left\{1,2,\ldots,k\right\}$ with the discrete topology. Let $\Sigma(k)$ be the product space $[k]^\ZZ$, whose elements are the sequences $\underline{a}=(\ldots,a_{-1},a_0,a_1,\ldots)$, with $a_n\in [k]$ forall $n\in\ZZ$. This space is endowed with the product topology, which is given by the metric
\[d(\underline{a},\underline{b})=\sum_{n=-\infty}^{\infty}\frac{\delta_n(\underline{a},\underline{b})}{2^{2\left|n\right|}}\]
where $\delta_n(\underline{a},\underline{b})$ is $0$ when $a_n=b_n$ and $1$ otherwise. The {\it shift} is a homeomorphism of $\Sigma(k)$ defined by
\[(\sigma(\underline{a}))_i=a_{i+1},\,\,\,i\in\ZZ\]
and has a special class of closed invariant sets: if $M_k$ is the set of $k\times k$ matrices with entries 0 or 1, for each $A\in M_k$, the set $$\Sigma_A=\{\underline{a}\in\Sigma(k):A_{a_i a_{i+1}}=1,\forall i\in\ZZ\}$$
is a closed invariant subspace of $\Sigma(k)$.

\begin{defi}\label{Shift}
The pair $(\Sigma_A,\sigma_A)$, where $\sigma_A=\sigma|_{\Sigma_A}$, is called a {\it subshift of finite type}.
\end{defi}


\subsection{Topological entropy}

Let $(X,d)$ be a metric space and $f:X\rightarrow X$ a uniformly continuous map. For every $n\in\NN$, define a new metric $d_n$ on $X$ by
\[d_n(x,y)=\max\{d(f^i(x),f^i(y)),i\in\{0,1,\ldots,n-1\}\}.\]
Let $B_{\delta}(x)$ and $\overline{B}_{\delta}(x)$ denote, respectively, the open and the closed ball centered at $x$ with radius $\delta$ in the metric $d$. The open ball centered at $x$ with radius $r$ in the metric $d_n$ is
\begin{eqnarray*}
B(n-1,r,x)&=&\{y\in K : d(f^j(x),f^j(y))< r,\forall j\in\{0,\ldots,n-1\}\} \\
&=&\bigcap_{i=0}^{n-1}f^{-i}(B_r(f^i(x)))
\end{eqnarray*}
while the closed ball is
\begin{eqnarray*}
\overline{B}(n-1,r,x)&=&\{y\in K : d(f^j(x),f^j(y))\leq r,\forall j\in\{0,\ldots,n-1\}\} \\
&=&\bigcap_{i=0}^{n-1}f^{-i}(\overline{B}_r(f^i(x))).
\end{eqnarray*}

\begin{defi}
Let $n\in\NN$, $\vep>0$ and $K$ be a compact subset of $X$. Given a subset $F$ of $X$, we say that $F$ $(n,\vep)-${\it spans $K$ with respect to $f$} if
\[\forall x\in K \,\,\exists y\in F:\,\,d_n(x,y)\leq\vep\]
or, equivalently,
\[K\subseteq\bigcup_{y\in F}\overline{B}(n-1,\vep,y).\]
\end{defi}

\begin{defi}
Let $n\in\NN$, $\vep>0$ and $K$ be a compact subset of $X$. Denote by $r_n(\vep,K)$ the smallest cardinality among all the $(n,\vep)$-spanning sets for $K$ with respect to $f$.
\end{defi}

Since $K$ is compact, we have $r_n(\vep,K)<\infty.$ Moreover,
$$\vep_1<\vep_2\Longrightarrow r_n(\vep_1,K)\geq r_n(\vep_2,K).$$

\begin{defi}
Let $\vep>0$ and $K$ be a compact subset of $X$. Define
\[r(\vep,K)=r(\vep,K,f)=\limsup_{n\rightarrow\infty}(1/n)\log r_n(\vep,K).\]
\end{defi}

\begin{defi}
Let $K$ be a compact subset of $X$. Define
$$h(f,K)=\lim_{\vep\rightarrow 0} r(\vep,K,f)$$
and the \textit{topological entropy} of $f$ as
$$h(f)=\sup\{h(f,K): K \text{ is a compact subset of } X\}.$$
\end{defi}

We will use an equivalent way of defining topological entropy which considers \emph{separated sets} instead of spanning sets.

\begin{defi}
Let $n\in\NN$, $\vep>0$ and $K$ be a compact subset of $X$. Given a subset $E$ of $K$, we say that \textit{$E$ is $(n,\vep)$-\textit{separated with respect to} $f$} if
\[\forall x,y\in E \,\,\, d_n(x,y)\leq\vep \,\,\Longrightarrow \,\,x=y\]
or, equivalently,
\[\forall x\in E \,\,\, \overline{B}(n-1,\vep,x)\cap E=\{x\}.\]
\end{defi}

\begin{defi}
Let $n\in\NN$, $\vep>0$ and $K$ be a compact subset of $X$. Denote by $s_n(\vep,K)$ the largest cardinality among all $(n,\vep)$-separated sets for $K$ with respect to $f$.
\end{defi}

We remark that $r_n(\vep,K) \leq s_n(\vep,K) \leq r_n(\vep/2,K).$ Moreover, since $r_n(\vep/2,K)<\infty$, we have $s_n(\vep,K)<\infty.$ Also,
$$\vep_1<\vep_2\Longrightarrow s_n(\vep_1,K)\geq s_n(\vep_2,K).$$

\begin{defi}
Let $\vep>0$ and $K$ be a compact subset of $X$. Define
\[s(\vep,K)=s(\vep,K,f)=\limsup_{n\rightarrow\infty}(1/n)\log s_n(\vep,K).\]
\end{defi}

Notice that $r(\vep,K) \leq s(\vep,K) \leq r(\vep/2,K).$ So
$$h(f,K)=\lim_{\vep\rightarrow 0} r(\vep,K)=\lim_{\vep\rightarrow 0} s(\vep,K)$$
and the topological entropy of $f$ may be estimated as
\[h(f)=\sup_K h(f,K)=\sup_K \lim_{\vep\rightarrow 0} s(\vep,K,f)\]
\noindent where $K$ is any compact subset of $X$. When $X$ is compact, this computation may be simplified.

\begin{prop}[\cite{Wa}]\label{M1}
If $(X,d)$ is a compact metric space and $f:X\rightarrow X$ is a continuous map, then
$$h(f)=h(f,X)=\lim_{\vep\rightarrow 0}\limsup(1/n)\log r_n(\vep,X)=\lim_{\vep\rightarrow 0}\limsup(1/n)\log s_n(\vep,X).$$
\end{prop}

\begin{exe}
\em A matrix $A\in M_k$ is said to be {\it irreducible} if
$$\forall i,j\in[k] \,\,\exists \,\,n\in\NN:(A^n)_{i j}>0.$$
In this case, by Perron-Frobenius Theorem, we know that $A$ has a non-negative simple eigenvalue $\lambda$ which is greater than the absolute value of all the other eigenvalues, so
$$max_{i\in[k]}\left|\lambda_i\right|=\lambda,$$
where $\lambda_1,\lambda_2,\ldots,\lambda_k$ are all the eigenvalues of $A$.

\ms

\begin{prop}[\cite{Wa}]\label{M2}
The entropy of the subshift of finite type $\sigma_A:\Sigma_A\rightarrow\Sigma_A$ associated to an irreducible matrix $A$ is $\log\lambda$, where $\lambda$ is the largest positive eigenvalue of A. In particular, the entropy of $\sigma:\Sigma(k)\rightarrow\Sigma(k)$ is $\log k$.
\end{prop}

\end{exe}

Instead of the previous definition of entropy, we could have used open covers. If $X$ is a compact topological space, $f:X\rightarrow X$ a continuous map and $\mathcal{A}$ a finite open cover of $X$, then the entropy of $f$ relative to $\mathcal{A}$ is given by the limit
$$h(f,\mathcal{A})=\lim_{n \rightarrow + \infty}\, \frac{1}{n}\,\log \left(H(\bigvee_{i=0}^{n-1}\,f^{-i}\, \mathcal{A})\right)$$
where $H(\bigvee_{i=0}^{n-1}\,f^{-i}\, \mathcal{A})$ is the number of sets in a finite subcover of $\bigvee_{i=0}^{n-1}\,f^{-i}\, \mathcal{A}$ with smallest cardinality. The topological entropy is then given by
$$h(f)= \sup_{\mathcal{A}}\,\,h(f,\mathcal{A}).$$
The equality between the two ways of defining topological entropy is due to the fact, proved in \cite{Wa}, that

\begin{prop}\label{TOPENTR1and2}
Let $f:X\rightarrow X$ be a continuous map of a compact metric space $(X,d)$. Given $\epsilon>0$ and the covers $\mathcal{B}$ and $\mathcal{C}$ of $X$ by open balls of radius $2\epsilon$ and $\frac{\epsilon}{2}$, respectively, then
$$H(\bigvee_{i=0}^{n-1}\,f^{-i}\, \mathcal{B})\leq r_n(\epsilon,X)\leq s_n(\epsilon,X)\leq H(\bigvee_{i=0}^{n-1}\,f^{-i}\, \mathcal{C}).$$
\end{prop}

\section{The zeta function}

Given a dynamical system $f$, let $N_n(f)$ be the total number of points for which $n$ is a period (not necessarily the smallest possible period), that is to say, the number of points $x$ for which $f^n(x)=x$, which we assume to be finite for all $n\in\NN$. The most natural measure of the asymptotic growth of these topological invariants is the exponential growth rate $\wp(f)$ (also called \emph{periodic entropy} of $f$) given by
\[\wp(f)=\limsup_{n\rightarrow\infty} \, \frac{\log(\max\{N_n(f),1\})}{n}.\]

One may join all the information given by the sequence $\left(N_n(f)\right)_{n\in\NN}$ in a single power series, the $\zeta$-function of $f$:
\[\zeta_f(z)=\exp\left(\sum_{n=1}^\infty \frac{N_n(f)}{n}z^n\right)\]
where $z$ is a complex number. Notice that, since the exponential is an entire function, the radius of convergence of $\zeta_f$ is
\[\rho=\frac{1}{\limsup\sqrt[n]{\frac{N_n(f)}{n}}}=\frac{1}{\limsup\sqrt[n]{N_n(f)}}.\]
If $f$ has no periodic points, then $\zeta_f=1$ and $\rho=\infty$. Otherwise, if $f$ has at least one periodic point, then
\[\exp(\wp(f))=\limsup_{n\rightarrow\infty}\sqrt[n]{\max\{N_n(f),1\}}=\limsup_{n\rightarrow\infty}\sqrt[n]{N_n(f)}=\frac{1}{\rho}\]
that is,
$$\rho=\exp(-\wp(f)).$$
If $\wp(f)<\infty$, that is to say, if the growth rate of the number of periodic points with the period is at most exponential, then this series has a positive radius of convergence. In fact, it converges for $|z|<\exp(-\wp(f))$ and always has singularities on the circle $|z|=\exp(-\wp(f))$.

\begin{exe}
\em If $f$ has only one periodic orbit, with period $p$, then
\begin{eqnarray*}
\zeta_f(z)&=&\exp \left(\sum_{n=1}^\infty \frac{N_n(f)}{n}z^n\right) \\
&=&\exp\left(z^p+\frac{z^{2p}}{2}+\cdots+\frac{z^{np}}{n}+\cdots\right) \\
&=&\exp \left(- \log(1-z^p)\right) \\
&=&\frac{1}{1-z^p}
\end{eqnarray*}
with radius of convergence equal to $1$. In general, expressing the set of periodic points of $f$ as a disjoint union of finite orbits $\mathcal{O}$ with periods $\mathcal{P}(\mathcal{O})$, we have
$$\zeta_f(z) = \prod_{\mathcal{O}} \,\left(1+z^{\mathcal{P}(\mathcal{O})}+z^{2\mathcal{P}(\mathcal{O})}+\cdots\right).$$
Thus the zeta function of $f$, if defined, is a formal power series with nonnegative integer coefficients.
\end{exe}

\begin{exe}
\em If there is a positive integer $\alpha$ such that, for all $n$, we have $N_n(f)=\alpha^n$, then
$$\zeta_f(z)=\exp\left(\sum_{n=1}^\infty\frac{(\alpha z)^n}{n}\right)=\exp\left(-\log(1-\alpha z)\right)=\frac{1}{1-\alpha z}$$
with radius of convergence equal to $\displaystyle \frac{1}{\alpha}$.
\end{exe}

In some cases, the series $\zeta_f$ actually represents a rational function of $z$, so the information it contains may be replaced by a finite set of numbers: the coefficients when $\zeta_f$ is written as a rational map. For instance, this happens when $f=\sigma_A$ (see Definition~\ref{Shift}): we can compute the zeta function, it is rational and $\wp(\sigma_A)$ is precisely the entropy of $f$.



\begin{prop}\label{M4}
$\zeta_{\sigma_A}(z)=\displaystyle \frac{1}{\det(I-zA)}.$
\end{prop}

\begin{proof}
Let $\lambda_1,\lambda_2,...,\lambda_k$ be the eigenvalues of $A$, so that
\[\det(zI-A)=(z-\lambda_1)(z-\lambda_2)\ldots(z-\lambda_k).\]
Replacing $z$ by $z^{-1}$, we get
\[\det(z^{-1}I-A)=(z^{-1}-\lambda_1)(z^{-1}-\lambda_2)\ldots(z^{-1}-\lambda_k)\]
and, multiplying both sides by $z^k$, we obtain
\[z^k\det(z^{-1}I-A)=z^k(z^{-1}-\lambda_1)(z^{-1}-\lambda_2)\ldots(z^{-1}-\lambda_k)\]
and so
\[\det(I-zA)=(1-\lambda_1 z)(1-\lambda_2 z)\ldots(1-\lambda_k z).\]

\ms

\begin{lem}[\cite{Shu}]\label{M3}
For all $n \in \mathbb{N}$, $N_n(\sigma_A)=\tr(A^n).$
\end{lem}

\ms

Since the eigenvalues of $A^n$ are $\lambda_1^n,\lambda_2^n,...,\lambda_k^n$, we have $\tr(A^n)=\sum_{m=1}^k \lambda_m^n$. Hence,
\[\zeta_{\sigma_A}(z)=\exp\left(\sum_{n=1}^\infty\frac{\sum_{m=1}^k \lambda_m^n}{n}z^n\right)=\exp\left(\sum_{m=1}^k\left(\sum_{n=1}^\infty\frac{(\lambda_m z)^n}{n}\right)\right).\]
As $\sum_{n=1}^\infty\frac{z^n}{n}=\log\left(\frac{1}{1-z}\right)$,
\begin{eqnarray*}
\zeta_{\sigma_A}(z) &=& \exp\left(\sum_{m=1}^k\log\left(\frac{1}{1-\lambda_m z}\right)\right) \\
&=& \exp\left(\log\left(\prod_{m=1}^k\left(\frac{1}{1-\lambda_m z}\right)\right)\right) \\
&=& \frac{1}{\prod_{m=1}^k(1-\lambda_m z)} \\
&=& \frac{1}{\det(I-zA)}.
\end{eqnarray*}
Thus, $\zeta_{\sigma_A}$ has no zeros, and its poles are the numbers $\frac{1}{\lambda_m}$, where $\{\lambda_1, \ldots, \lambda_k\}$ is the set of eigenvalues of the matrix $A$.
\end{proof}

\begin{exe}
\em Let $A=\left(\begin{array}{cc}1&1\\1&0\end{array}\right)$. The eigenvalues of $A$ are $\lambda_1=\frac{1+\sqrt{5}}{2}$ and $\lambda_2=\frac{1-\sqrt{5}}{2}$, so
\[\zeta_{\sigma_A}(z)=\frac{1}{(1-\lambda_1 z)(1-\lambda_2 z)}=\frac{1}{1-z-z^2}\]
with radius of convergence equal to $\frac{1}{\lambda_1}$.
\end{exe}

\ms

\begin{prop}\label{M5}
Let $A$ be an irreducible matrix with entries 0 or 1. Then the topological entropy of $\sigma_A$ is equal to $\wp(\sigma_A)=-\log\rho$, where $\rho$ is the radius of convergence of $\zeta_{\sigma_A}$.
\end{prop}

\begin{proof}
Since $\zeta_{\sigma_A}(z)=1/\det(I-zA)$ and
\[
\det(I-zA)=0\Leftrightarrow\prod_{m=1}^k(1-\lambda_m z)=0\Leftrightarrow \exists m\in[k]:z=1/\lambda_m \wedge \lambda_m\neq 0,
\]
the radius of convergence of $\zeta_{\sigma_A}$ is given by
\[
\rho=\min\left\{\left|1/\lambda_i\right|: i\in [k] \wedge \lambda_i\neq 0\right\}=1/\max\left\{\left|\lambda_i\right|: i\in [k] \wedge \lambda_i\neq 0\right\}=1/\lambda.
\]
Therefore $\wp(\sigma_A)=-\log\rho=\log\lambda$ is the topological entropy of $\sigma_A$. (So, in this case, we have $\wp\left(\sigma^n_A\right)=|n|\,\wp\left(\sigma_A\right)$, for all $n \in \mathbb{Z}$.)
\end{proof}

\subsection{Expansive maps}

Let $(X,d)$ be a metric space and $f:X\rightarrow X$ a continuous map.

\begin{defi}
We say that $\vep$ is an {\it expansivity constant} for $f$ if
\[d(f^n(x),f^n(y))\leq\vep,\,\,\,\forall n\in\NN_0\,\,\Longrightarrow \,\,x=y.\]
The map $f$ is said to be {\it expansive} if there is an {\it expansivity constant} for $f$.
\end{defi}

Notice that, if $f$ is expansive and $X$ is compact, then, for any $n \in \mathbb{N}$, the periodic points with period $n$ are isolated. In fact, as $f$ is uniformly continuous, we may associate, to the constant of expansivity $\vep$, a positive $\delta$ such that, for all $0\leq j < n$ and all $x, y \in X$,
$$d(x,y)<\delta \,\,\,\Rightarrow \,\,\,d(f^j(x),f^j(y)) < \vep \,\,\,\,\, \forall \, 0 \leq j < n.$$
If $p$ and $q$ are two distinct periodic points with period $n$, then, by the expansivity, there exists $t \in \mathbb{N}_0$ such that $d(f^t(p), f^t(q))\geq \vep$; as $f^n(p)=p$ and $f^n(q)=q$, such a $t$ may be chosen in $\{0,1,\ldots,n-1\}$; therefore we must have $d(p,q) \geq \delta$. And so, as $X$ is compact, the set of periodic points with period $n$ is finite, for all $n \in \mathbb{N}$.

\begin{prop}\label{M6}
If $(X,d)$ is a compact metric space and $f:X\rightarrow X$ is expansive, then $N_n(f)<\infty$, for all $n\in\NN$, and $\zeta_f$ has a positive radius of convergence.
\end{prop}

\begin{proof}
Suppose that $f$ is a continuous map with expansivity constant $\vep$. Let $U_1,\ldots,U_r$ be a cover of $X$ with $diam(U_i)\leq\vep,\forall i\in\left[r\right]$ (notice that we can take $r=r_1(\vep,X)$). For each $x\in X$, let $\phi(x)=(a_0,a_1,a_2,\ldots)$, with $a_n=\min\{i\in\left[r\right]:f^n(x)\in U_i\}$. We can see that $$\phi(x)=\phi(y)\,\,\Rightarrow \,\,d(f^n(x),f^n(y))\leq\vep,\,\,\,\forall n\in\NN_0\,\,\Rightarrow \,\,x=y,$$
so $\phi$ is injective. Also, if $x$ is periodic with period $n$, then so is $\phi(x)$. Since the number of periodic points in $\left[r\right]^{\NN_0}$ with period $n$ is $r^n$, we have $N_n(f)\leq r^n<\infty$ and
\[\wp(f)=\limsup_{n\rightarrow\infty} \, \frac{\log(\max\{N_n(f),1\})}{n} \leq\log r\]
so
$$\rho \geq 1/r>0.$$
\end{proof}

\begin{cor} For all $z$ such that $\left|z\right|<1/r$, we have
$$ 1-r\left|z\right| \leq \left|\zeta_f(z)\right| \leq \frac{1}{1-r\left|z\right|}.$$
\end{cor}

\begin{proof}
\begin{eqnarray*}
\left|\zeta_f(z)\right| &=& \left|\exp\left(\sum_{n=1}^\infty\frac{N_n(f)}{n}z^n\right)\right| = \exp\left(\sum_{n=1}^\infty\frac{N_n(f)}{n}Re(z^n)\right) \\
&\leq& \exp\left(\sum_{n=1}^\infty\frac{r^n}{n}\left|z^n\right|\right) = \exp\left(\sum_{n=1}^\infty\frac{(r\left|z\right|)^n}{n}\right) \\
&=& \exp\left(\log\left(\frac{1}{1-r\left|z\right|}\right)\right) = \frac{1}{1-r\left|z\right|}
\end{eqnarray*}
and, similarly,
$$\left|\zeta_f(z)\right| = \exp\left(\sum_{n=1}^\infty\frac{N_n(f)}{n}Re(z^n)\right) \geq \exp\left(\sum_{n=1}^\infty\frac{r^n}{n}(-\left|z^n\right|)\right)= 1-r\left|z\right|$$
for all $z$ such that $\left|z\right|<1/r$ (recall that $\rho\geq 1/r$).
\end{proof}

\noindent \textbf{Remark}: There are closed invariant subsets of $\Sigma(k)$ for which the zeta function for the restriction of $\sigma$ to those sets is not rational (see \cite{BL} for details).\\

\subsection{Hyperbolic $\mathcal{C}^1$ diffeomorphisms}

Let $f$ be a $\mathcal{C}^1$ diffeomorphism defined on a hyperbolic set with local product structure. The map $f$ is expansive (see \cite{Shu}), so $N_n(f)<\infty$ for all $n\in\NN$, and we can define the zeta function for $f$. Moreover, as proved in \cite{M} (see also \cite{Shu}),

\begin{theo}\label{M7}
The zeta function of a $\mathcal{C}^1$ diffeomorphism on a hyperbolic set with local product structure is rational.
\end{theo}

As a consequence, if $f$ is a $\mathcal{C}^1$ diffeomorphism such that $\overline{Per(f)}$ is hyperbolic, then $\zeta_f$ is a rational function: in fact, it is known that, if $\overline{Per(f)}$ is hyperbolic, then it has a local product structure \cite{Shu}; and $\zeta_f=\zeta_{f|_{\overline{Per(f)}}}$. In particular, if $f$ is Axiom A, then $\zeta_f$ is rational.\\

The main ingredient of the known argument to prove this Theorem is the existence of a Markov partition of arbitrarily small diameter, which allows one to establish a codification of most of the orbits of $f$ through a subshift of finite type (for which we already know how to count the periodic points) and to translate the properties of the zeta function from the subshift to the diffeomorphism setting.

\begin{exe}
\em
If $f$ is the linear toral endomorphism induced by an integer matrix M, then the number of fixed points for $f^n$ is $N_n(f)=|det(M^n-1)|$ (see \cite{BRW}). In particular, if $f$ is a hyperbolic automorphism, then $N_n(f)=\varsigma^n (\tr(M^n)-1-\det(M)^n)$, where $\varsigma=sgn(\tr(M))$; thus
\[
\zeta_f(z)=\frac{(1-\varsigma \, z)(1-\varsigma\,\det(M)\,z)}{\det(I-\varsigma\, M \, z)}=\frac{(1-\varsigma \, z)(1-\varsigma\,\det(M)\,z)}{1-|tr(M)|z+\det(M)z^2}
\]
which is a rational function with integer coefficients.\\

For instance, if $M=\left(
\begin{array}{cc}
2 & 1 \\
1 & 1
\end{array}
\right)$, then
$$N_n(f)=\left(\frac{3+\sqrt{5}}{2}\right)^n+\left(\frac{3-\sqrt{5}}{2}\right)^n -2$$
and so
\[
\zeta_f(z)=\frac{(1-z)^2}{\left(1-(\frac{3+\sqrt{5}}{2})z\right)\left(1-(\frac{3-\sqrt{5}}{2})z\right)}=\frac{(1-z)^2}{1-3z+z^2}
\]
with radius of convergence equal to $\frac{2}{3+\sqrt{5}}$.
\end{exe}

\subsection{Ruelle-expanding maps}\label{RE}

Let $(K,d)$ be a compact metric space and $f:K\rightarrow K$ a continuous map.

\begin{defi}\label{R-expanding}
$f$ is {\it Ruelle-expanding} if there are $r>0$, $0<\lambda<1$ and $c>0$ such that:
\begin{itemize}
\item $\forall x,y \in K, \,\,x\neq y \wedge f(x)=f(y)\,\,\Longrightarrow \,\,d(x,y)>c$
\item $\forall x \in K,\,\,\forall a \in f^{-1}(\{x\}),\,\,\exists \,\,\phi:B_r(x)\rightarrow K$ verifying
$$\phi(x)=a$$
$$(f\circ\phi)(y)=y,\,\,\,\forall y \in B_r(x)$$
$$d(\phi(y),\phi(z))\leq \lambda d(y,z),\,\,\,\forall y,z\in B_r(x).$$
\end{itemize}
\end{defi}

\begin{exe}
\em Let $M$ be a compact Riemannian manifold without boundary and consider a $\mathcal{C}^1$ map $f:M\rightarrow M$. One says that $f$ is {\it expanding} if
$$\exists\,\,\lambda\in\left]0,1\right[:\,\,\forall x\in M,\,\,\left\|D_x f(v)\right\|\geq 1/\lambda \left\|v\right\|.$$
It is easy to prove that, in the $\mathcal{C}^1$ context, $f$ is expanding if and only if it is Ruelle-expanding. More details about this family of maps may be found in \cite{Ru3}. One example of such a map is
\begin{eqnarray*}
f:S^1&\rightarrow&S^1\\
z&\mapsto&z^k
\end{eqnarray*}
with $k>1$ a positive integer. It is the lifting to $S^1$ of the piecewise expanding map
\begin{eqnarray*}
T:[0,1] & \rightarrow & [0,1]\\
t & \mapsto & kt \,\, \text{ mod } 1
\end{eqnarray*}
it is expanding, with $\lambda=1/k$, and its topological entropy is equal to $\log k$. This map has $k^n-1$ periodic points with period $n$ and so its $\zeta$-function is equal to
\begin{eqnarray*}
\zeta_f(z)&=&\exp\left(\sum_{n=1}^\infty \frac{k^n-1}{n}z^n\right) \\
&=&\exp\left(-\log(1-kz)+\log(1-z)\right) \\
&=&\frac{1-z}{1-kz}
\end{eqnarray*}
which is a rational function, with radius of convergence equal to $\frac{1}{k}$.

More generally, if $L:\mathbb{R}^n \rightarrow \mathbb{R}^n$ is a linear map whose eigenvalues have absolute value bigger than one and such that $L(\mathbb{Z}^n)\subseteq \mathbb{Z}^n$, then $L$ induces in the flat torus $\mathbb{R}^n / \mathbb{Z}^n$ a Ruelle-expanding map. (Conversely, any $\mathcal{C}^1$ expanding map in the $n$-dimensional flat torus is topologically conjugate to one obtained by this process \cite{Shu3}.)

\end{exe}

\begin{exe}
\em Let $\Sigma(k)^+$ be the product space $[k]^{\NN_0}$, whose elements are the sequences $\underline{a}=(a_0,a_1,\ldots)$, with $a_n\in [k],\forall n\in\NN_0$, endowed with the product topology which can be generated by the metric given by $d(\underline{a},\underline{b})=\sum_{n=0}^{\infty}\frac{\delta_n(\underline{a},\underline{b})}{2^n}$, where $\delta_n(\underline{a},\underline{b})$ is $0$ when $a_n=b_n$ and $1$ otherwise. The dynamics in $\Sigma(k)^+$, called \textit{unilateral} (or \textit{one-sided}) \textit{shift}, is defined as $(\sigma^+(\underline{a}))_i=a_{i+1},\,\,i\in\NN_0$. For each $A\in M_k$, consider $\Sigma_A^+=\{\underline{a}\in\Sigma(k)^+:A_{a_i a_{i+1}}=1,\forall i\in\NN_0\}$. The pair $(\Sigma_A^+,\sigma_A^+)$, where $\sigma_A^+=\sigma^+|_{\Sigma_A^+}$, is called a {\it unilateral subshift of finite type}. $\sigma_A^+$ is Ruelle-expanding, with $r=1$ and $\lambda=c=1/2$:

\begin{itemize}
\item If $\underline{a}\neq\underline{b}$ and $\sigma_A^+(\underline{a})=\sigma_A^+(\underline{b})$, then $a_0\neq b_0$, so $d(\underline{a},\underline{b})\geq 1>c$.
\item If $r=1$, then, for any $\underline{a}\in\Sigma_A^+$ we have $B_r(\underline{a})=\{\underline{b}\in\Sigma_A^+:b_0=a_0\}$ since, as we have seen, $b_0\neq a_0\Rightarrow d(\underline{a},\underline{b})\geq 1=r$. Also, the pre-images of $\underline{a}=(a_0,a_1,a_2,\ldots)$ are of the form $(x,a_0,a_1,\ldots)$, where $A_{x a_0}=1$. If we define $\phi(\underline{b})=(x,b_0,b_1,b_2,\ldots)$ for $\underline{b}=(b_0,b_1,b_2,\ldots)\in B_r(\underline{a})$ (that is to say, with $a_0=b_0$), then $\sigma_A^+(\phi(\underline{b}))=\underline{b}$ and, for any $\underline{b},\underline{c}\in B_r(\underline{a})$, we have
\[d(\phi(\underline{b}),\phi(\underline{c}))=\sum_{n=1}^{\infty}\frac{\delta_{n-1}(\underline{b},\underline{c})}{2^n}=\sum_{n=0}^{\infty}\frac{\delta_n(\underline{b},\underline{c})}{2^{n+1}}=\frac{d(\underline{b},\underline{c})}{2}=\lambda d(\underline{b},\underline{c}).\]
\end{itemize}

If we take $\sigma=\sigma_A^+$, then $\underline{a}\in\Sigma_A^+$ is a fixed point of $\sigma^n$ if and only if $a_i=a_{i+n},\forall i\in\NN_0$. To each fixed point of $\sigma^n$, given by $\underline{a}=(a_0,a_1,a_2,...,a_0,a_1,a_2,...)$, we can associate a unique admissible sequence of length $n+1$ defined by $a_0 a_1 a_2 ... a_{n-1} a_0$. So the number of fixed points of $\sigma^n$ is $N_n(\sigma)=\tr(A^n)$ and so $\zeta_{\sigma}(z)=\frac{1}{\det(I-z A)}$, also a rational function. The full one-sided shift is just a particular case of a subshift of finite type, with $A_{i j}=1,\forall i,j\in [k]$, and its zeta function is $\zeta_{\sigma}(z)=\frac{1}{1-kz}$.
\end{exe}

\ms

\noindent \textbf{Remark}: The dynamics of the circle map $f(z)=z^k$ is essentially the one of the full one-sided shift $\sigma$ defined on $\Sigma(k)^+$. However, the semiconjugacy between these two dynamical systems maps two distinct fixed points of $\sigma$ (more precisely $(1,1,1,\ldots)$ and $(k,k,k,\ldots)$) into the same (and unique) fixed point of $f$. This explains the difference between $\zeta_f(z)=\frac{1-z}{1-kz}$ and $\zeta_{\sigma}(z)=\frac{1}{1-kz}$.

\ms

\noindent \textbf{Remark}: Similarly to what happens with the bilateral subshift, the topological entropy of the one-side subshift of finite type $\sigma_A^+:\Sigma_A^+\rightarrow\Sigma_A^+$ associated to an irreducible matrix $A$ is $\log\lambda$, where $\lambda$ is the largest positive eigenvalue of A. Since the radius of convergence $\rho$ of $\zeta_{\sigma_A^+}$ is given by $1/\lambda$ (the argument is identical to the one used in the two-sided subshift setting), we conclude that the topological entropy of the subshift of finite type is $-\log\rho$ in both cases. That is, topological and periodic entropies are equal in this setting. Moreover, the probability measure of maximal entropy is the weak* limit of the sequence $\left(\nu_n\right)_{n \in \mathbb{N}}$ defined, for each $n \in \mathbb{N}$, by
$$\nu_n=\displaystyle \frac{1}{N_n(f)}\,\sum_{x \,\in \,\text{Per}_n(f)}\, \delta_x$$
where $f=\sigma_A^+$ or $f=\sigma_A$ (details in \cite{Wa}).

\ms

\begin{defi}\label{contractive branch}
Let $f:K\rightarrow K$ be Ruelle-expanding and $S\subseteq K$. Given $n\in\NN$, we say that $g:S\rightarrow K$ is a {\it contractive branch} of $f^{-n}$ if
\begin{itemize}
	\item $(f^n\circ g)(x)=x, \forall x \in S$
	\item $d((f^j\circ g)(x),(f^j\circ g)(y))\leq \lambda^{n-j}d(x,y),\,\,\forall x,y\in S,j\in\{0,1,\ldots,n\}.$
\end{itemize}
\end{defi}

\bigskip

It is easy to see (\cite{Ru1},\cite{Cra}) that, given $x\in K$, $n\in\NN$ and $a\in f^{-n}(\{x\})$, there is always a contractive branch $g:B_r(x)\rightarrow K$ of $f^{-n}$ with $g(x)=a$. Moreover,

\ms

\begin{prop}\label{M8}
There is $\vep_0<r$ such that, for every $\vep$ with $0<\vep<\vep_0$, we have:
\begin{itemize}
	\item[(a)] $\forall n\in\NN$, $B(n,\vep,x)=g(B_{\vep}(f^n(x)))$, where $g:B_r(f^n(x))\rightarrow K$ is a contractive branch of $f^{-n}$ with $g(f^n(x))=x.$
	\item[(b)] $\vep$ is an expansivity constant for $f.$
\end{itemize}
\end{prop}

\ms

\begin{proof}

\noindent $(a)$ Consider $\vep_0=\min\, \{r, \frac{c}{1+\lambda}\}$ and $0<\vep<\vep_0$.

\noindent The inclusion $B(n,\vep,x)\supseteq g(B_{\vep}(f^n(x)))$ is valid by definition of contractive branch. Conversely, for $n=1$, take $x \in K$ and $z \in B(1,\vep,x)$. Then $d(z,x)\leq \vep$ and $d(f(z),f(x))\leq \vep$. So, if $g:B_r(f(x))\rightarrow K$ is the map $\phi$ obtained in Definition~\ref{R-expanding} using $f(x)$ and $a=x \in f^{-1}(\{f(x)\})$, then $g(f(x))=x$ and
$$d(g \circ f(z),x) = d(g \circ f(z),g \circ f(x))\leq \lambda \, \vep.$$
Therefore,
$$d(z, g\circ f(z)) \leq d(z,x) + d(x, g\circ f(z)) \leq \vep + \lambda \, \vep = (1+\lambda)\vep < c.$$
As $f(z)=f(g(f(z)))$, we must have $z=g(f(z))$, which proves that $B(1,\vep,x)\subseteq g(B_{\vep}(f^n(x)))$. The argument proceeds by induction.

\bigskip

\noindent $(b)$ If $d(f^n(x), f^n(y))\leq \vep$ for all $n \in \mathbb{N}_0$, then, using item $(a)$, we deduce that $d(x,y)\leq \lambda^n\,\vep$, for all $n \in \mathbb{N}_0$, and so $x=y$.
\end{proof}

\begin{prop}[\cite{Ru1},\cite{Cra}]\label{M9}
$K=\bigcup_{n\geq 0}f^{-n}(\overline{Per(f)})$, where $Per(f)$ is the set of periodic points of $f$. In particular, $Per(f)\neq\emptyset$.
\end{prop}

Notice that, since $f$ is expansive, its zeta function has a positive radius of convergence. Also, as $f$ has at least one periodic point,
$$\rho=\exp(-\wp(f))\leq 1.$$

\section{Examples}

The existence of a differentiable expanding map is a nontrivial topological restriction on the compact manifold. For instance, among orientable compact surfaces without boundary, only the torus possesses such kind of maps. In general, the set of $\mathcal{C}^1$ expanding maps defined on a connected compact flat manifold is non-empty (\cite{ES}). The fact now proved that the $\zeta$-function of an expanding map is rational evinces another instance of rigidity in the sense that, for some $k$, the first $k$ numbers of the sequence $\left(N_n(f)\right)_{n\in\NN}$ determine all the others.

\begin{cor}\label{M24}
Given a $\mathcal{C}^1$ expanding map on a compact Riemannian manifold, there are constants $k\in\NN_0$, $\ell\in\NN$, $(\gamma_i)_{1\leq i\leq k}$, $(n_i)_{1\leq i\leq k}\in\NN$, $(\eta_j)_{1\leq j\leq\ell}$ and $(m_j)_{1\leq j\leq\ell}\in\NN$ such that
\[
N_n(f)=\sum_{j=1}^{\ell}\frac{m_j}{(\eta_j)^{n}}-\sum_{i=1}^{k}\frac{n_i}{(\gamma_i)^{n}}.
\]
\end{cor}

\begin{proof}
As $\zeta_f$ is rational (but not a polynomial) and does not vanish at $z=0$, it has $k\geq 0$ zeros, say $(\gamma_i)_{1\leq i\leq k}$, with multiplicity $(n_i)_{1\leq i\leq k}\in\NN$ and $\ell\geq 1$ poles, say $(\eta_j)_{1\leq j\leq\ell}$, with multiplicity $(m_j)_{1\leq j\leq\ell}\in\NN$. Hence there is a constant $C$ such that
\[
\zeta_f(z)=C \,\frac{\prod_{i=1}^{k}(z-\gamma_i)^{n_i}}{\prod_{j=1}^{\ell}(z-\eta_j)^{m_j}}.
\]
Taking the logarithmic derivative of both the presentations of the zeta function, we get
\begin{eqnarray*}
\sum_{n=0}^\infty N_{n+1}(f)z^n &=& \sum_{i=1}^{k} \frac{n_i}{z-\gamma_i}-\sum_{j=1}^{\ell}\frac{m_j}{z-\eta_j} \\
&=& \sum_{n=0}^\infty \left(\sum_{j=1}^{\ell}\frac{m_j}{(\eta_j)^{n+1}}- \sum_{i=1}^{k}\frac{n_i}{(\gamma_i)^{n+1}}\right)z^n.
\end{eqnarray*}
And so, collating coefficients with the same degree, we deduce the explicit formula for the number of periodic points with period $n$ of $f$.
\end{proof}

According to \cite{JL} and \cite{Tau}, the set of periods for expanding maps defined on torus or flat compact manifolds are \emph{uniformly cofinite}, that is to say, there is a positive integer $m_0$, which depends only on the dimension of the manifold, such that, for all integers $m\geq m_0$, any expanding map on the manifold has a periodic point whose minimum period is exactly $m$. This means that the poles and zeros of the zeta functions of such maps have to obey strong restrictions to ensure that, for $m \geq m_0$, the difference
\[N_m(f)-\sum_{d|m \, , \, d<m} N_d(f)\]
is positive.

\ms

\noindent \textbf{Remark}: We have considered maps which are continuous and locally uniformly expanding, but these are not necessary conditions for the rationality of the $\zeta$ function. There are examples of maps defined on a closed interval whose $\zeta$-functions are rational, including some which are not continuous (although uniformly expanding) and some simultaneously not continuous and not uniformly expanding. For instance,
\begin{itemize}
\item The map
$$x \in \,[0,1] \mapsto f(x)=2x\, \text{ mod } \,1$$
is locally expanding, as required in the definition of Ruelle-expanding functions (with $c=\lambda=\frac{1}{2}$), but it is not continuous. The corresponding $\zeta$-function is rational because the restriction of the dynamics to the invariant set $D=[0,1]\setminus \{\text{dyadic rational numbers}\}$ is conjugated to a full unilateral subshift of finite type and only a fixed point is left outside of $D$.

\item The map
$$x \in \,[0,1] \mapsto f_s(x)=x+x^{1+s} \, \text{ mod} \,1$$
where $s$ is a positive constant, is not continuous and is not uniformly expanding (it even has a fixed point, at $0$, with first derivative equal to $1$). Nevertheless, its $\zeta$-function is rational because, similarly, there is an invariant domain $E\subseteq [0,1]$ with only a finite number of periodic points outside of it and such that there is a conjugacy between $f_s|_E$ and a restriction of a unilateral subshift of finite type.
\end{itemize}

\section{Entropy}

\subsection{Entropy \emph{vs.} radius of convergence}

Is there any relation between the radius of convergence $\rho$ and $h(f)$ for Ruelle-expanding maps? Indeed, we have
$$\wp(f)\leq h(f)$$
and so $\rho\geq\exp(-h(f))$. To prove this, we will see first how to simplify the computation of $h(f)$ in this context.

\begin{prop}\label{M10}
Let $f:K \rightarrow K$ be a Ruelle-expanding map on a compact metric space $(K,d)$, $\vep$ an expansivity constant for $f$ and $\mathcal{A}$ a finite cover of $K$ by open balls with radius smaller than $\vep/2$. Then:
\begin{itemize}
\item $h(f)=r(\vep_0,K)=s(\vep_0,K)$ for all $\vep_0<\vep/4$.
\item $h(f)=h(f, \mathcal{A})$.
\end{itemize}
\end{prop}

\begin{proof}
See \cite{Wa}. Although the proof in this reference is done for expansive homeomorphisms, it can be easily adapted for expansive maps.
\end{proof}

\ms

Let $p$ and $q$ be periodic points of $f$, with $f^n(p)=p$ and $f^n(q)=q$ for some $n\in\NN$. Then we have
\begin{eqnarray*}
d_n(p,q)\leq\vep_0\Longrightarrow d_n(p,q)\leq\vep&\Longrightarrow& d(f^i(p),f^i(q))\leq\vep,\forall i\in\{0,1,\ldots,n-1\}\\
&\Longrightarrow& d(f^i(p),f^i(q))\leq\vep,\,\,\forall i\in\NN_0\,\,\Longrightarrow p=q.
\end{eqnarray*}
So the set $Per_n(f)$ of periodic points $p$ with $f^n(p)=p$ is a $(n,\vep)$-separated set for $K$. Consequently,

\begin{cor}\label{PerEntr}
If $f:K\rightarrow K$ is a Ruelle-expanding map defined on a compact metric space $(K,d)$, then:
\begin{itemize}
\item $s_n(\vep_0,K)\geq s_n(\vep,K) \geq card\left(Per_n(f)\right)=N_n(f).$
\item $\wp(f)\leq h(f)$.
\end{itemize}
\end{cor}

\subsection{Entropy \emph{vs.} pre-images}

The entropy of $f$ is also related with the number of pre-images of the points in $K$ for $f$.

\begin{prop}\label{M11}
If $(K,d)$ is a compact metric space and $f:K\rightarrow K$ is a Ruelle-expanding map, then there is some $k\in\NN$ such that $card(f^{-1}(\{x\}))\leq k,\forall x\in K$.
\end{prop}

\begin{proof}
If we set $E=f^{-1}(\{x\})$, then
$$f(u)=f(v)=x,\,\,\forall u,v\in E,\,\,u\neq v,$$
thus $d_1(u,v)=d(u,v)>c$ and $E$ is a $(1,c)$-separated set. Since $card(E)\leq s_1(c,K)<\infty$, we can take $k=s_1(c,K)$.
\end{proof}

\begin{cor}\label{LOGK}
$h(f)\leq \log k$, so $0\leq \wp(f)\leq \log k$ and $1/k\leq\rho\leq 1.$
\end{cor}

\begin{proof}
Let $\vep_0<min\{\vep/4,c,r\}$. Since $K$ is compact, there is a finite set $F$ for which we can write
\[K=\bigcup_{y\in F}\overline{B}_{\vep_0}(y).\]
Given $x\in K$ and $n\in\NN$, let $y\in F$ be such that $d(f^n(x),y)\leq\vep_0$ and let $g:B_r(f^n(x))\rightarrow K$ be a contractive branch of $f^{-n}$ with $g(f^n(x))=x$. If we take $z=g(y)$, we get
$$f^n(z)=f^n(g(y))=y\Longrightarrow z\in f^{-n}(F)$$
and
\begin{eqnarray*}
d(f^i(x),f^i(z)) &=& d(f^i(g(f^n(x))),f^i(g(y))) \\
&\leq& \lambda^{n-i}d(f^n(x),y) \\
&\leq& \lambda^{n-i}\vep_0 \\
&\leq& \vep_0,\,\,\,\,\,\forall i\in\{0,1,\ldots,n-1\} \\
&\Longrightarrow& d_n(x,z)\leq\vep_0.
\end{eqnarray*}
So $f^{-n}(F)$ is a $(n,\vep_0)$-spanning set for $K$. Therefore
$$r_n(\vep_0,K)\leq card(f^{-n}(F))\leq k^n card(F),\,\,\,\forall n\in\NN$$
and we deduce that
\begin{eqnarray*}
h(f) &=& r(\vep_0,K) \\
&=& \limsup(1/n)\log r_n(\vep_0,K) \\
&\leq& \limsup(1/n)\log(k^n card(F)) \\
&=& \limsup(\log k+(1/n)\log(card(F))) \\
&=& \log k.
\end{eqnarray*}
So, $0\leq \wp(f)\leq \log k$ and $1/k\leq\rho\leq 1.$
\end{proof}

\begin{cor}
If there exists for some $k\in\NN$ such that $card(f^{-1}(\{x\}))=k$ for all $x\in K$, then $h(f)=\log k$.
\end{cor}

\begin{proof}
Fix $x\in K$ and take $E_n=f^{-n}(\{x\})$; then we have $f^n(u)=f^n(v)=x,\forall u,v\in E_n,u\neq v$. If $f(u)=f(v)$, then $d_n(u,v)\geq d(u,v)>c$; otherwise, we have $f(u)\neq f(v)$. Admitting the latter, if $f^2(u)=f^2(v)$, then $d_n(u,v)\geq d(f(u),f(v))>c$, otherwise, we have $f^2(u)\neq f^2(v)$. Proceeding this way, since we have $f^n(u)=f^n(v)$, there must be some $j\in\{1,\ldots,n\}$ for which $f^j(u)=f^j(v)$ and $f^{j-1}(u)\neq f^{j-1}(v)$, so $d_n(u,v)\geq d(f^{j-1}(u),f^{j-1}(v))>c$ and $E_n$ is a $(n,c)$-separated set. Since $card(E_n)=k^n$, we have $k^n\leq s_n(c,K)\leq s_n(\vep_0,K)$ and therefore we get
\[h(f)=s(\vep_0,K)=\limsup(1/n)\log s_n(\vep_0,K)\geq\limsup(1/n)\log(k^n)=\log k\]
which, with the estimate of the previous Corollary, allow us to conclude that, in this particular case, $h(f)=\log k$.
\end{proof}

\begin{exe}
\em Let $M$ be a compact Riemannian manifold and $f:M\rightarrow M$ an H\"{o}lder $\mathcal{C}^1$ expanding map. Then $card(f^{-1}(x))$ is independent of $x\in M$; it is called the \emph{degree} of $f$ and denoted by $deg(f)$. Moreover, as we have seen, $f$ is Ruelle-expanding. So $h(f)=\log(deg(f))$.
\end{exe}

\section{Proof of Theorem~\ref{MainTheorem1}}\label{Rational}

Our aim now is to prove the rationality of the zeta function for Ruelle-expanding maps. Recall that the existence of a Markov partition was an essential ingredient in the proof of the rationality of the zeta function for $\mathcal{C}^1$ diffeomorphisms defined on a hyperbolic set with local product structure. In the case of Ruelle-expanding maps, we will prove the existence of finite covers with analogous properties, which will play the same role the Markov partitions did.

\begin{prop}\label{M12}
Let $f$ be a Ruelle-expanding map defined on a compact set $K$. Denote by $\vep$ an expansivity constant for $f$. Then $K$ has a finite cover $\{R_1,...,R_k\}$ with the following properties:
\begin{itemize}
\item Each $R_i$ has a diameter less than $min\{\vep,c/2\}$.
\item Each $R_i$ is proper, that is to say, it is equal to the closure of its interior.
\item $\stackrel{\circ}{R_i}\cap\stackrel{\circ}{R_j}=\emptyset,\,\,\,\forall i,j\in[k],\,\,i\neq j.$
\item $f(\stackrel{\circ}{R_i})\cap\stackrel{\circ}{R_j}\neq\emptyset\,\,\Longrightarrow \,\,\stackrel{\circ}{R_j}\subseteq f(\stackrel{\circ}{R_i}).$
\end{itemize}
\end{prop}

\ms

\noindent \textbf{Remark}: If $\stackrel{\circ}{R_j}\subseteq f(\stackrel{\circ}{R_i})$, then $R_j=\overline{\stackrel{\circ}{R_j}}\subseteq \overline{f(\stackrel{\circ}{R_i})}\subseteq f\left(\overline{\stackrel{\circ}{R_i}}\right)=f(R_i)$ and the last condition implies that $f(\stackrel{\circ}{R_i})\cap\stackrel{\circ}{R_j}\neq\emptyset\Longrightarrow R_j\subseteq f(R_i)$.

\bigskip

To prove this Proposition, we will begin by a Shadowing Lemma. Given $\alpha > 0$ and a map $f:K\rightarrow K$, we say that the sequence $(x_n)_{n\in\NN_0}$ is an $\alpha$-pseudo orbit if, for any $n\in\NN_0$, we have $d(f(x_n),x_{n+1})<\alpha$. This sequence admits a $\beta$-shadow in $K$, for some $\beta > 0$, if there exists a point $x\in K$ such that $d(f^n(x),x_n)<\beta$ for all $n\in\NN_0$.

\begin{lem}\label{M13}
Let $f:K\rightarrow K$ be Ruelle-expanding defined on a compact space $K$. For any $\beta\in\,\,]0,r[$ there is $\alpha>0$ such that, if $(x_n)_{n\in\NN_0}$ is an $\alpha$-pseudo orbit in $K$, then it admits a $\beta$-shadow in $K$. Besides, the $\beta$-shadow is unique if $\beta<\vep/2$, where $\vep$ is an expansivity constant for $f$.
\end{lem}

\begin{proof}
Firstly we will prove this statement for finite $\alpha$-pseudo orbits. Let $\beta\in\,\,]0,r[$ and $(x_0,x_1,\ldots,x_n)$ be such that $d(f(x_{k-1}),x_k)<\alpha,\forall k\in[n]$, for some $\alpha>0$. If $y_n=x_n$, then $d(y_n,x_n)=0<\beta$. Now, suppose that $d(y_k,x_k)<\beta$ for $k\in[n]$. Since $d(f(x_{k-1}),x_k)<\alpha$, we have $d(y_k,f(x_{k-1}))<\alpha+\beta<r$, if we assume that $\alpha<r-\beta$. Then we can take $y_{k-1}=g(y_k)$, where $g:B_r(f(x_{k-1}))\rightarrow K$ is a contractive branch of $f^{-1}$ with $g(f(x_{k-1}))=x_{k-1}$; thus we get $d(y_{k-1},x_{k-1})\leq \lambda d(y_k,f(x_{k-1}))<\lambda (\alpha+\beta)<\beta$, if we assume that $\alpha<\frac{1-\lambda}{\lambda}\beta$. Also, notice that $y_k=f(y_{k-1}),\forall k\in[n]$, so that $y_k=f^k(x),\forall k\in[n]$, for $x=y_0$. Hence, it is enough to take $\alpha<\min\{r-\beta,\frac{1-\lambda}{\lambda}\beta\}$.

Now, take $\beta\in]0,r[$ and let $(x_n)_{n\in\NN_0}$ be an $\alpha$-pseudo orbit, with $\alpha<\min\{\frac{r-\beta}{2},\frac{1-\lambda}{2\lambda}\beta\}$. Let $z_n$ be a $\beta/2$-shadow of $(x_0,x_1,\ldots,x_n)$; since $K$ is compact, there is some subsequence $(z_{n_k})_k$ converging to a point $z\in K$. As $d(f^i(z_{n_k}),x_i)<\beta/2,\forall i\in\{0,1,\ldots,n_k\}$, we deduce that, for fixed $i\in\NN_0$,
$$d(f^i(z),x_i)=\lim_{k \rightarrow +\infty} d(f^i(z_{n_k}),x_i)\leq\beta/2<\beta,$$
and so $z$ is a $\beta$-shadow of $(x_n)_{n\in\NN_0}$.

Concerning the uniqueness of the $\beta$-shadow when $\beta<\vep/2$, suppose that $z$ and $z'$ are both $\beta$-shadows of $(x_n)_{n\in\NN_0}$. Then we have $$d(f^i(z),f^i(z'))\leq d(f^i(z),x_i)+d(x_i, f^i(z'))<2\beta<\vep$$
for all $i\in\NN_0$, and so $z=z'$.
\end{proof}

\ms

In particular,

\begin{cor}\label{PER}
Let $f:K\rightarrow K$ be a Ruelle-expanding map defined on a compact metric space $K$, with an expansivity constant $\vep$. For any $\beta<\vep/2$, there is $\alpha_\beta >0$ such that, if $x \in K$ verifies $d(f^p(x),x)< \alpha_\beta$, then there exists a unique periodic point $z \in K$ such that $f^p(z)=z$ and $d(f^j(x),f^j(z))<\beta$ for all $0\leq j \leq p$.
\end{cor}

\begin{proof}
Define $x_i=f^k(x)$ for $i\equiv k \text{ mod } p$, where $k \in [0,p[$. Then $(x_i)_{i \in \mathbb{N}_0}$ is an $\alpha$-pseudo orbit. If $z$ is the (unique) $\beta$ shadow of it, then, for all $i \in \mathbb{N}_0$,
$$d(f^i(z),f^i(f^p(z)))\leq d(f^i(z),x_i)+d(x_i, f^{i+p}(z))= d(f^i(z),x_i)+d(x_{i+p}, f^{i+p}(z))<2\beta<\vep$$
and so, by the expansivity of $f$, we obtain $f^p(z)=z$.
\end{proof}

Fix $\vep$ be an expansivity constant for $f$ with $\vep<r$ and some $\beta<\min\{\vep/2,c/4\}$. Let $\alpha$ be given by Lemma~\ref{M13} and $\gamma\in\,\,]0,\alpha/2[$ be such that
$$d(x,y)<\gamma\,\,\,\Rightarrow \,\,\,d(f(x),f(y))<\alpha/2,\,\,\,\,\forall x,y\in K.$$
Since $K$ is compact, we can take $\{p_1,\ldots,p_k\}$ such that $K=\bigcup_{i=1}^k B_{\gamma}(p_i)$. We define a matrix $A\in M_k$ by

\begin{center}
$A_{ij}=1\,\,\,$ if $\,\,\,d(f(p_i),p_j)<\alpha$ $\,\,\,\,$ and $\,\,\,\,A_{ij}=0$ otherwise.\\
\end{center}
For every $\underline{a}\in\Sigma^+_A$, the sequence $(p_{a_i})_{i\in\NN_0}$ is an $\alpha$-pseudo orbit, so it admits a unique $\beta$-shadow which we will denote by $\theta(\underline{a})$. In this way we have defined a map $\theta:\Sigma_A^+\rightarrow K$ verifying:

\begin{lem}\label{M14}
$\theta$ is a semiconjugacy between $\sigma^+_A$ and $f$.
\end{lem}

\begin{proof}
Given $x\in K$, we can take $a_i\in [k]$ so that $d(f^i(x),p_{a_i})<\gamma$ for any $i\in\NN_0$. Then
$$d(f(p_{a_i}),p_{a_{i+1}})\leq d(f(p_{a_i}),f(f^i(x)))+d(f^{i+1}(x),p_{a_{i+1}})<\alpha/2+\gamma<\alpha$$
confirming that $(p_{a_i})_{i\in\NN_0}$ is an $\alpha$-pseudo orbit. Therefore $x=\theta(\underline{a})$ and $\theta$ is surjective.

To prove the continuity, since $K$ is compact, it suffices to see that, for any two sequences $(\underline{s}^n)_{n\in\NN}$ and $(\underline{t}^n)_{n\in\NN}$ converging to the same limit $l$ in $\Sigma_A^+$ whose images under $\theta$ converge respectively to $s$ and $t$ in $K$, we have $s=t$.
Fix some $i\in\NN_0$; for any $n\in\NN$, we have $d(f^i(\theta(\underline{s}^n)),p_{s^n_i})<\beta$ and $d(f^i(\theta(\underline{t}^n)),p_{t^n_i})<\beta$. So taking limits we have $d(f^i(s),p_{l_i})\leq\beta$ and $d(f^i(t),p_{l_i})\leq\beta$. Hence, $d(f^i(s),f^i(t))\leq 2\beta<\vep$ and, since $\vep$ is an expansivity constant for $f$, we get $s=t$.

The relation $f\circ\theta=\theta\circ\sigma^+_A$ is a consequence of the uniqueness of the $\beta$-shadow and the fact that, if $x$ is a $\beta$-shadow for $(p_{a_i})_i$, then $f(x)$ is a $\beta$-shadow for $(p_{a_{i+1}})_i=(p_{\sigma_A^+(a_i)})_i$.
\end{proof}

\ms

Let $T_i=\{\theta(\underline{a}):a_0=i\}$ for $i\in [k]$. The set $T_i$ is closed since $C_i$ is compact and $\theta$ is continuous. Moreover $T_i=\theta(C_i)$ where $C_i=\{\underline{a}\in\Sigma_A^+:a_0=i\}$, and, since $\Sigma_A^+=\bigcup_{i=1}^k C_i$, we have $K=\bigcup_{i=1}^k T_i$ because $\theta$ is surjective. Hence, $\{T_i,i\in[k]\}$ is a finite closed cover of $K$.

\begin{lem}\label{M15}
If $A_{ij}=1$, then $T_j\subseteq f(T_i)$ and $\stackrel{\circ}{T_j}\subseteq f(\stackrel{\circ}{T_i})$. Also, given $x\in T_i$ with $f(x)\in T_j$, if $g:B_r(f(x))\rightarrow K$ is a contractive branch of $f^{-1}$ with $g(f(x))=x$, then $g(T_j)\subseteq T_i$ and $g(\stackrel{\circ}{T_j})\subseteq \,\stackrel{\circ}{T_i}$.
\end{lem}

\begin{proof}
Given any $y\in T_j$, we have $y=\theta(\underline{b})$ for some $\underline{b}\in\Sigma_A^+$ with $b_0=j$. Since $A_{ij}=1$, we can take $\underline{c}=(i,b_0,b_1,b_2,\ldots)\in \Sigma_A^+$, and so $y=\theta(\underline{b})=\theta(\sigma_A^+(\underline{c}))=f(\theta(\underline{c}))\in f(\theta(C_i))=f(T_i)$. Then $T_j\subseteq f(T_i)$.

Notice that $T_j\subseteq B_\beta(p_j)$. Since $d(f(x),p_j)<\beta$, we have $T_j\subseteq B_{2\beta}(f(x))\subseteq B_r(f(x))$. Let $g:B_r(f(x))\rightarrow K$ be a contractive branch of $f^{-1}$ with $g(f(x))=x$. Given $y\in T_j$, we have $y=f(z)$ for some $z\in T_i$. Thus
\[
d(g(y),z)\leq d(g(y),g(f(x)))+d(x,p_i)+d(p_i,z)<d(y,f(x))+2\beta<4\beta<c
\]
and, since $f(g(y))=y=f(z)$, we get $g(y)=z\in T_i$. So $g(T_j)\subseteq T_i$.

It is easy to see that $g:B_r(f(x))\rightarrow g(B_r(f(x)))$ is a homeomorphism, with
$$g^{-1}=f|_{g(B_r(f(x)))}:g(B_r(f(x)))\rightarrow B_r(f(x)).$$
Therefore we conclude that $g(\stackrel{\circ}{T_j})=\overbrace{g(T_j)}^{\circ}\subseteq\,\stackrel{\circ}{T_i}$ and $\stackrel{\circ}{T_j}=f(g(\stackrel{\circ}{T_j}))\subseteq f(\stackrel{\circ}{T_i})$.
\end{proof}

Let $Z=K\backslash\bigcup_{i=1}^k \partial T_i$. Notice that, since $T_i$ is a closed set, $\partial T_i$ has empty interior. So $Z$ is dense in $K$. Given $x\in Z$, we define

\begin{center}
$T_i^*(x)=\,\stackrel{\circ}{T_i}\,\,$ if $\,\,x\in\,\stackrel{\circ}{T_i}\,\,$ and $\,\,T_i^*(x)=K\backslash T_i\,\,$ if $\,\,x\notin T_i$
\end{center}

\begin{center}
$R(x)=\bigcap_{i=1}^k T_i^*(x)$
\end{center}
The sets $R(x)$ satisfy the following properties:

\begin{itemize}

\item  $R(x)$ is open.

(because it is a finite intersection of open sets)

\item  $x\in R(x).$

(because $x\in T_i^*(x),\forall i\in[k]$)

\item  $R(x)\subseteq\,\stackrel{\circ}{T_i}$ for some $i\in[k]$.

Since $\bigcap_{i=1}^k K\backslash T_i=K\backslash\bigcup_{i=1}^k T_i=\emptyset$, we must have $x\in\,\stackrel{\circ}{T_i}$ for some $i\in[k]$.\\

\item  If $R(x)\cap R(y)\neq\emptyset$, then $R(x)=R(y)$.

In fact,
\begin{eqnarray*}
R(x)\cap R(y)\neq\emptyset &\Rightarrow& \forall i\in[k],T_i^*(x)\cap T_i^*(y)\neq\emptyset \\
&\Rightarrow& \forall i\in[k],T_i^*(x)=T_i^*(y) \\
&\Rightarrow& R(x)=R(y).
\end{eqnarray*}
\end{itemize}

\begin{lem}\label{M16}
Given $x\in Z\cap f^{-1}(Z)$ and a contractive branch $g:B_r(f(x))\rightarrow K$ of $f^{-1}$ with $g(f(x))=x$, we have $g(R(f(x)))\subseteq R(x)$.
\end{lem}

\begin{proof}
Let $y\in R(f(x))$. Notice that $y\in Z$ and $f(x)\in R(y)$. For $i\in[k]$, if $x\in T_i$, then $x=\theta(\underline{a})$ for some $\underline{a}\in\Sigma_A^+$ with $a_0=i$. Let $j=a_1$. Then $f(x)=\theta(\sigma(\underline{a}))$ and $f(x)\in T_j$, so that $y\in R(f(x))\subseteq T_j\Rightarrow g(y)\in g(T_j)$. Since $A_{ij}=1$, by  Lemma~\ref{M15} we get $g(T_j)\subseteq T_i$ and, hence, $g(y)\in T_i$.

On the other hand, if $g(y)\in T_i$ then $g(y)=\theta(\underline{b})$ for some $\underline{b}\in\Sigma_A^+$ with $b_0=i$. Let $j=b_1$. Then $y=f(g(y))=\theta(\sigma(\underline{b}))$ and $y\in T_j$, so that $f(x)\in R(y)\subseteq T_j\Rightarrow x=g(f(x))\in g(T_j)$. Since $A_{ij}=1$, by Lemma~\ref{M15} we get $g(T_j)\subseteq T_i$ and, hence, $x\in T_i$. So $x\in T_i\Leftrightarrow g(y)\in T_i,\forall i\in[k]$.

Similarly, using Lemma~\ref{M15} we obtain $x\in\,\stackrel{\circ}{T_i}\Leftrightarrow g(y)\in \,\stackrel{\circ}{T_i},\forall i\in[k]$, and so conclude that $g(y)\in R(x)$.
\end{proof}

Let $R=\{\overline{R(x)},x\in Z\}$. Since $R$ is a finite set, we can write $R=\{R_1,\ldots,R_k\}$, with $R_i\neq R_j$ if $i\neq j$, for some $k\in\NN$. Also, since $Z$ is dense in $K$, we have $K=\overline{\bigcup_{x\in Z}\{x\}}=\overline{\bigcup_{x\in Z}R(x)}=\bigcup_{x\in Z}\overline{R(x)}=\bigcup_{i=1}^s R_i$, that is to say, $R$ is a finite closed cover of $K$. Let us see that $R$ satisfies the other required properties.

\begin{itemize}
\item $R_i$ has a diameter less than $min\{\vep,c/2\}$ and is proper.

Take $x\in Z$ such that $R_i=\overline{R(x)}$ and $j\in[k]$ such that $R(x)\subseteq\,\stackrel{\circ}{T_j}$. Then $R_i=\overline{R(x)}\subseteq\overline{\stackrel{\circ}{T_j}}\subseteq \overline{T_j}=T_j$ and $diam(R_i)\leq diam(T_j)\leq 2\beta<min\{\vep,c/2\}$. Also, taking into account that the closure of the interior of the closure of the interior of a set is just the closure of the interior of that set, we have
$$\overline{\stackrel{\circ}{R_i}}=\overline{\stackrel{\circ}{\overline{R(x)}}}=\overline{\stackrel{\circ}{\overline{\stackrel{\circ}{R(x)}}}}=\overline{\stackrel{\circ}{R(x)}}=\overline{R(x)}=R_i$$ because $R(x)$ is open.\\

\item $\stackrel{\circ}{R_i}\cap\stackrel{\circ}{R_j}=\emptyset,\forall i,j\in[k],i\neq j.$

Take $x,y\in Z$ such that $R_i=\overline{R(x)}$ and $R_j=\overline{R(y)}$. Suppose that $\stackrel{\circ}{R_i}\cap\stackrel{\circ}{R_j}\neq\emptyset$; using the fact that any open set that intersects the closure of a set also intersects the set itself, we get
\begin{eqnarray*}
\stackrel{\circ}{\overline{R(x)}}\cap\stackrel{\circ}{\overline{R(y)}}\neq\emptyset &\Rightarrow& \stackrel{\circ}{\overline{R(x)}}\cap\overline{R(y)}\neq\emptyset \\
&\Rightarrow& \stackrel{\circ}{\overline{R(x)}}\cap R(y)\neq\emptyset \\
&\Rightarrow& \overline{R(x)}\cap R(y)\neq\emptyset \\
&\Rightarrow& R(x)\cap R(y)\neq\emptyset \\
&\Rightarrow& R(x)=R(y) \\
&\Rightarrow& R_i=R_j \\
&\Rightarrow& i=j.
\end{eqnarray*}

\item $f(\stackrel{\circ}{R_i})\cap\stackrel{\circ}{R_j}\neq\emptyset \Rightarrow \stackrel{\circ}{R_j}\subseteq f(\stackrel{\circ}{R_i}).$

Since $f$ takes open sets into open sets and $Z$ is dense in $K$, $f^{-1}(Z)$ is also dense in $K$. Besides, $Z$ is a nonempty open set, so $Z\cap f^{-1}(Z)$ is dense in $Z$, and, hence, $Z\cap f^{-1}(Z)$ is dense in $K$. Since $\stackrel{\circ}{R_i}\cap f^{-1}(\stackrel{\circ}{R_j})$ is a nonempty open set, we have $Z\cap f^{-1}(Z)\cap\stackrel{\circ}{R_i}\cap f^{-1}(\stackrel{\circ}{R_j})\neq\emptyset$, so we can take $x\in Z\cap\stackrel{\circ}{R_i}$ with $f(x)\in Z\cap\stackrel{\circ}{R_j}$.
Notice that
$$x\in R(x)\subseteq\,\stackrel{\circ}{\overline{R(x)}}\Longrightarrow \stackrel{\circ}{R_i}\cap\stackrel{\circ}{\overline{R(x)}}\neq\emptyset\Longrightarrow R_i =\overline{R(x)}$$
and, similarly, that $R_j=\overline{R(f(x))}$. Using Lemma~\ref{M16} and the fact that $g$ is conti\-nuous, we get $$g(R_j)=g(\overline{R(f(x))})\subseteq\overline{g(R(f(x)))}\subseteq\overline{R(x)}=R_i\Longrightarrow R_j=f(g(R_j))\subseteq f(R_i).$$

\end{itemize}


\ms

We may now construct a semiconjugacy between $f$ and a unilateral subshift of finite type. Let $\left\{R_1,...,R_k\right\}$ be a cover of $K$ like above. As usual, we define a matrix $A\in M_k$, which encodes the itineraries of the orbits by $f$ inside the partition, by

\begin{center}
\it $A_{i j}=1\,\,$ if $\,\,f(\stackrel{\circ}{R_i})\cap\stackrel{\circ}{R_j}\neq\emptyset\,\,$ and $\,\,A_{i j}=0$ otherwise.
\end{center}

\begin{lem}\label{M17}
Let $(a_0,...,a_n)$ be an admissible sequence for $A$. Then $\bigcap_{i=0}^n f^{-i}(\stackrel{\circ}{R_{a_i}})\neq\emptyset$.
\end{lem}

\begin{proof}
The statement is trivial for sequences with just one element. Suppose now that the assertion is valid for the admissible sequence $(a_1,...,a_n)$, so that $\bigcap_{i=0}^{n-1} f^{-i}(\stackrel{\circ}{R_{a_{i+1}}})\neq\emptyset$. Let $y\in\bigcap_{i=0}^{n-1}f^{-i}(\stackrel{\circ}{R_{a_{i+1}}})$. Since $A_{a_0 a_1}=1$, we have $\stackrel{\circ}{R_1}\subseteq f(\stackrel{\circ}{R_0})$. So $y=f(x)$ for some $x\in\,\stackrel{\circ}{R_0}$ and it is easy to see that $x\in\bigcap_{i=0}^n f^{-i}(\stackrel{\circ}{R_{a_i}})$.
\end{proof}

As a consequence of Lemma~\ref{M17}, we can see that, for each sequence $\underline{a}=(a_n)_{n\in\NN_0}\in\Sigma_A^+$, if $F_n=\bigcap_{i=0}^n f^{-i}(R_{a_i})$ then $(F_n)_n$ is a nested sequence of nonempty compact sets, so its limit is nonempty. Besides, if $x$ and $y$ are two points in this intersection, then $\forall i\in\NN_0,d(f^i(x),f^i(y))\leq diam(R_{a_i})<\vep$, so $x=y$. Therefore we may define a map $\Pi:\Sigma_A^+\rightarrow K$ as
\[\{\Pi(\underline{a})\}=\lim_{n \rightarrow +\infty} \, F_n=\bigcap_{n=0}^\infty f^{-n}(R_{a_n}).\]

\begin{lem}
$\Pi$ is a semiconjugacy between $\sigma_A^+$ and $f$.
\end{lem}

\begin{proof}

Let $\underline{a}\in\Sigma_A^+$. Notice that $f(f^{-1}(L))\subseteq L$ for any $L\subseteq K$. Therefore
\[
\{f(\Pi(\underline{a}))\}=f \left(\bigcap_{n=0}^\infty f^{-n}(R_{a_n})\right)\subseteq
f \left(\bigcap_{n=1}^\infty f^{-n}(R_{a_n})\right)=
\]

\[
=f \left(f^{-1}\left(\bigcap_{n=1}^\infty f^{-(n-1)}(R_{a_n})\right)\right)\subseteq
\bigcap_{n=0}^\infty f^{-n}(R_{a_{n+1}})=\{\Pi(\sigma_A^+(\underline{a}))\}
\]
So $f(\Pi(\underline{a}))=\Pi(\sigma_A^+(\underline{a}))$. As $\Pi$ is also surjective and continuous, it semiconjugates $\sigma_A^+$ and $f$. (So $h(f)\leq h(\sigma_A^+)\leq \log k$.)
\end{proof}

\ms

Since $\Pi$ is not necessarily injective, a point in $K$ can have more than one preimage under $\Pi$. However, we will show that it cannot have more than $k$ pre-images. (Recall that $k$ is the number of elements of the covering we are dealing with.)

\begin{lem}\label{M18}
Let $(a_0,...,a_n)$ and $(b_0,...,b_n)$ be two admissible sequences for $A$ with $a_n=b_n$. If, for any $i\in\{0,\ldots,n\}$, we have $R_{a_i}\cap R_{b_i}\neq\emptyset$, then the sequences are equal.
\end{lem}

\begin{proof}
We have seen in Lemma~\ref{M17} that $\bigcap_{i=0}^n f^{-i}(\stackrel{\circ}{R_{a_i}})\neq\emptyset$, so there is some $x\in K$ with $f^i(x)\in\,\stackrel{\circ}{R_{a_i}}$.
By hypothesis, $R_{a_n}=R_{b_n}$. Suppose now that, for $i\in[n]$, we have $R_{a_i}=R_{b_i}$. Since $A_{a_{i-1} a_i}=A_{b_{i-1} b_i}=1$, we get $\stackrel{\circ}{R_{a_i}}\subseteq f(\stackrel{\circ}{R_{a_{i-1}}})$ and $\stackrel{\circ}{R_{b_i}}\subseteq f(\stackrel{\circ}{R_{b_{i-1}}})$. Then, since $f^i(x)\in\,\stackrel{\circ}{R_{a_i}}=\,\stackrel{\circ}{R_{b_i}}$, there are $y\in\,\stackrel{\circ}{R_{a_{i-1}}}$ and $z\in\,\stackrel{\circ}{R_{b_{i-1}}}$ such that $f^i(x)=f(y)=f(z)$. Also, $d(y,z)\leq diam(R_{a_{i-1}})+diam(R_{b_{i-1}})\leq c$ because $R_{a_{i-1}}\cap R_{b_{i-1}}\neq\emptyset$. So $y=z$ and $\stackrel{\circ}{R_{a_{i-1}}}\cap\stackrel{\circ}{R_{b_{i-1}}}\neq\emptyset$. Since different elements of the partition must have disjoint interior, we conclude that $R_{a_{i-1}}=R_{b_{i-1}}$.
\end{proof}

\begin{prop}\label{M19}
Any point of $K$ has no more than $k$ pre-images under $\Pi$.
\end{prop}

\begin{proof}
Suppose, by contradiction, that there was a point in $x\in K$ with $k+1$ distinct pre-images. Call these pre-images $\underline{x}^1,\underline{x}^2,\ldots,\underline{x}^{k+1}$. Then, for $n$ big enough, the admissible sequences $(x_0^i,\ldots,x_n^i)$ must be different from each other. But, since we have $k+1$ sequences, at least two of them must have the same last element, so they should be equal by Lemma~\ref{M18}. (Remind that, by definition of $\Pi$, for every $m\in\{0,\ldots,n\}$ and $i\in\left[k+1\right]$, we have $f^m(x)\in R_{x_m^i}$.)
\end{proof}

\begin{prop}\label{M20}
The pre-images of periodic points of $f$ are periodic points of $\sigma_A^+$.
\end{prop}

\begin{proof}
To simplify the notation, denote $\sigma=\sigma_A^+$. Assume that $x\in K$ is such that $f^p(x)=x$ for some $p\in\NN$. Let $\underline{x}^1,\underline{x}^2,\ldots,\underline{x}^r$ be the pre-images of $x$, distinct from each other by hypothesis. Then, for every $i\in[r]$, we have $\Pi(\sigma^p(\underline{x}^i))=f^p(\Pi(\underline{x}^i))=f^p(x)=x$, so that $\sigma^p(\underline{x}^1),\sigma^p(\underline{x}^2),\ldots,\sigma^p(\underline{x}^r)$ are also pre-images of $x$.

Suppose that there are $i,j\in[r]$, $i\neq j$, with $\sigma^p(\underline{x}^i)=\sigma^p(\underline{x}^j)$; in particular, we have $x_p^i=x_p^j$. Then the admissible sequences $(x_0^i,\ldots,x_p^i)$ and $(x_0^j,\ldots,x_p^j)$ verify the hypothesis of Lemma~\ref{M18}, and so they must be equal. Thus $$\underline{x}^i=(x_0^i,x_1^i\ldots,x_p^i,x_{p+1}^i,\ldots)=(x_0^j,x_1^j\ldots,x_p^j,x_{p+1}^j,\ldots)=\underline{x}^j,$$
which contradicts the assumption that $\underline{x}^1,\underline{x}^2,\ldots,\underline{x}^r$ are distinct from each other.

Then $\sigma^p(\underline{x}^1),\sigma^p(\underline{x}^2),\ldots,\sigma^p(\underline{x}^r)$ are also distinct from each other and, therefore, they are precisely the pre-images of $x$. So there is a permutation $\mu\in S_r$ such that $\sigma^p(\underline{x}^i)=\underline{x}^{\mu(i)}$ for every $i\in[r]$. Hence $\sigma^{ord(\mu)p}(\underline{x}^i)=\underline{x}^{\mu^{ord(\mu)}(i)}=\underline{x}^i$ for every $i\in[r]$.
\end{proof}

\ms

In spite of the existence of the semiconjugacy $\Pi$ between $\sigma_A^+$ and $f$, we may have $N_p(f)\neq N_p(\sigma_A^+)$, mainly for two reasons:
\begin{itemize}
\item If two rectangles intersect at their boundaries, there the map $\Pi$ is many to one, and so several points in $\text{Fix}\left((\sigma_A^+)^p\right)$ may be mapped to the same point in $\text{Fix}(f^p)$.
\item The map $f^p$ may rotate its domain in such a way that two rectangles are interchanged while their common boundary is kept fixed. In that case, a periodic point by $f$ with period $p$ belonging to that boundary would correspond, through $\Pi$, to points with higher period by $\sigma_A^+$, say $2p$.
\end{itemize}

To capture all these events that affect the estimation of the number of periodic points of $f$, we will construct subshifts whose alphabets are sets of $r\in[k]$ intersecting rectangles, using an algebraic device to cancel out the overcounting. \\

For each $r\in[k]$, consider
\[
I_r=\left\{\{s_1,\ldots,s_r\}\subset[k]:\bigcap_{i=1}^r R_{s_i}\neq\emptyset\right\}
\]
where we assume that $s_1<s_2<\ldots<s_r$. Let $A^{(r)}$ and $B^{(r)}$ be matrices with coefficients indexed by the set $I_r$ and defined as follows:

\begin{defi}
Given $s,t\in I_r$, with $s=\{s_1,...,s_r\}$ and $t=\{t_1,...,t_r\}$, if there is a unique permutation $\mu\in S_r$ such that $A_{s_i t_{\mu(i)}}=1$ for every $i\in[r]$, then
$$A^{(r)}_{s t}=1\,\, \text{ and } \,\,B^{(r)}_{s t}=sgn(\mu)$$
where $sgn(\mu)$ denotes the signature of the permutation $\mu$ (equal to $1$ if the permutation is even and to $-1$ if it is odd); otherwise, set
$$A^{(r)}_{s t}=B^{(r)}_{s t}=0.$$
\end{defi}

\noindent \textbf{Remark}: $A^{(1)}=A$.

\ms

Let $\Sigma_r^+=I_r^{\NN_0}$ be the set of sequences indexed by $\NN_0$ whose elements belong to $I_r$ and $\Sigma(A^{(r)})^+\subseteq\Sigma_r^+$ be the subset of admissible sequences according to the matrix $A^{(r)}$. Besides, let $\sigma_r^+$ denote the unilateral shift defined on these sets.

If $x\in Per_p(f)$, let $\underline{\alpha}^1,\ldots,\underline{\alpha}^r$ be the pre-images of $x$ under the map $\Pi$ (notice that $r\leq k$, by Proposition~\ref{M19}).
Then:

\begin{lem}\label{M21}
If $\underline{s}$ and $\underline{t}$ are two pre-images of $x$ by $\Pi$ with $s_i=t_i$ for some $i\in\NN_0$, then $\underline{s}=\underline{t}$.
\end{lem}

\begin{proof}
In fact, since $\underline{s}$ and $\underline{t}$ are both periodic points, there is some common period $n$, so that $\sigma^n(\underline{s})=\underline{s}$ and $\sigma^n(\underline{t})=\underline{t}$. Then the sequences $(s_i,s_{i+1},\ldots,s_{i+n})$ and $(t_i,t_{i+1},\ldots,t_{i+n})$ verify the hypothesis of Lemma~\ref{M18}: they end with the same element ($s_{i+n}=s_i=t_i=t_{i+n}$) and, by definition of $\Pi$, $f^m(x)\in R_{s_m}$ and $f^m(x)\in R_{t_m}$ for every $m\in\{i,\ldots,i+n\}$.
\end{proof}

For each $m\in\NN_0$ and $i\in[r]$, we have $f^m(x)\in R_{\alpha_m^i}$, so $\bigcap_{i\in[r]}R_{\alpha_m^i}\neq\emptyset$ and, since $\alpha_m^i\neq\alpha_m^j$ for $i\neq j$ (by Lemma~\ref{M21}), we can define an element $\hat{a}_m\in I_r$ and, therefore, build a sequence $\underline{\hat{a}}=(\hat{a}_m)_{m\in\NN_0}\in\Sigma_r^+$.\\

We will now see how to induce, through $\Pi$, a map $\hat{\Pi}_r:\Sigma(A^{(r)})^+\rightarrow K$. Given a sequence $\underline{\hat{a}}=(\hat{a}_n)_n\in\Sigma(A^{(r)})^+$, with $\hat{a}_n=\{a_n^1,...,a_n^r\}\in I_r$, for every $n\in\NN_0$, there is, by definition of $\Sigma(A^{(r)})^+$, a unique permutation $\mu_n$ such that $A_{a_n^i a_{n+1}^{\mu_n(i)}}=1,\,\,\forall i\in[r]$. Consider the permutations
\begin{center}
$\nu_0=id$

$\nu_n=\mu_{n-1}\circ\ldots\circ\mu_1\circ\mu_0.$
\end{center}
Notice that $\mu_n\circ\nu_n=\nu_{n+1}$, for all $n\in\NN_0$. For each $i\in[r]$ and $m\in\NN_0$, let $\alpha_m^i=a_m^{\nu_m(i)}$. Then $\underline{\alpha}^i=\left(\alpha_m^i\right)_m$ belongs to $\Sigma_A^+$, for every $i\in[r]$. In fact, we have, for all $m\in\NN_0,$
\[
A_{\alpha_m^i\alpha_{m+1}^i}=A_{a_m^{\nu_m(i)} a_{m+1}^{\nu_{m+1}(i)}}=A_{a_m^{\nu_m(i)}a_{m+1}^{\mu_m(\nu_m(i))}}=1.
\]
We know that, for every $m\in\NN_0$, there is some $y_m\in\bigcap_{i=1}^r R_{a_m^i}$ because $\hat{a}_m\in I_r$. So, for all $i,j\in[r]$, we have
\begin{eqnarray*}
d(f^m(\Pi(\underline{\alpha}^i)),f^m(\Pi(\underline{\alpha}^j))) &\leq& d(f^m(\Pi(\underline{\alpha}^i)),y_m)+d(y_m,f^m(\Pi(\underline{\alpha}^j))) \\
&\leq& 2 \, \max_{n\in[k]}\left\{diam (R_n)\right\} \\
&<& \delta <\vep/2
\end{eqnarray*}
which implies that $\Pi(\underline{\alpha}^i)=\Pi(\underline{\alpha}^j)$. Then, for each $r\in[k]$, we can define a map $\hat{\Pi}_r:\Sigma(A^{(r)})^+\rightarrow K$ by setting $\hat{\Pi}_r(\underline{\hat{a}})=\Pi(\underline{\alpha}^i)$, which does not depend on the choice of the index $i\in[r]$.

Let us verify that $\hat{\Pi}_r(Per_p(\sigma_r^+))\subseteq Per_p(f)$. Given $\underline{\hat{a}}\in Per_p(\sigma_r^+)$, we have
\[
\{\hat{\Pi}_r(\underline{\hat{a}})\}=\{\Pi(\underline{\alpha}^i)\}=\bigcap_{n\in\NN_0}f^{-n}(R_{\alpha_n^i})
\]
for any $i\in[r]$. So
\[
\{\hat{\Pi}_r(\underline{\hat{a}})\}=\bigcap_{i\in[r]}\bigcap_{n\in\NN_0}f^{-n}(R_{\alpha_n^i})=\bigcap_{n\in\NN_0}f^{-n}\left(\bigcap_{i\in[r]}R_{\alpha_n^i}\right)=\bigcap_{n\in\NN_0}f^{-n}\left(\bigcap_{i\in[r]}R_{a_n^i}\right)
\]
and
\begin{eqnarray*}
\{f^p(\hat{\Pi}_r(\underline{\hat{a}}))\} &=& f^p\left(\bigcap_{n\in\NN_0}f^{-n}\left(\bigcap_{i\in[r]}R_{a_n^i}\right)\right)\subseteq\bigcap_{n\in\NN_0}f^{p-n}\left(\bigcap_{i\in[r]}R_{a_n^i}\right) \\
&\subseteq& \bigcap_{n\in\NN_0,n\geq p}f^{p-n}\left(\bigcap_{i\in[r]}R_{a_n^i}\right)=\bigcap_{n\in\NN_0}f^{-n}\left(\bigcap_{i\in[r]}R_{a_{n+p}^i}\right) \\
&=& \bigcap_{n\in\NN_0}f^{-n}\left(\bigcap_{i\in[r]}R_{a_n^i}\right)=\{\hat{\Pi}_r(\underline{\hat{a}})\}
\end{eqnarray*}
because $\hat{a}_n=\hat{a}_{n+p},\forall n\in\NN_0$. Hence, $f^p(\hat{\Pi}_r(\underline{\hat{a}}))=\hat{\Pi}_r(\underline{\hat{a}})$.

Furthermore, $\mu=id$ is the only permutation in $S_r$ such that $A_{\alpha_m^i \alpha_{m+1}^{\mu(i)}}=1,\forall i\in[r]$. In fact, take a permutation $\mu\in S_r$, with order $\tau$, such that $A_{\alpha_m^i \alpha_{m+1}^{\mu(i)}}=1,\forall i\in[r]$. Given any $j\in[r]$, consider the two admissible sequences
\[
\alpha_n^j\alpha_{n+1}^{\mu(j)}\cdots\alpha_{n+q}^{\mu(j)}\alpha_{n+q+1}^{\mu^2(j)}\cdots\alpha_{n+(\tau-1)q}^{\mu^{\tau-1}(j)}\alpha_{n+(\tau-1)q+1}^j
\]
and
\[
\alpha_n^j\alpha_{n+1}^j\cdots\alpha_{n+q}^j\alpha_{n+q+1}^j\cdots\alpha_{n+(\tau-1)q+1}^j
\]
where $q$ is a common period of the pre-images of $x$. By Lemma~\ref{M18}, they must be equal; in particular, $\alpha_{n+1}^{\mu(j)}=\alpha_{n+1}^j$. Then Lemma~\ref{M21} tells us that $\mu(j)=j$ and, therefore, $\mu=id$.

In this way we have deduced that $\underline{\hat{a}}\in\Sigma(A^{(r)})^+$. Also, as we have seen before, the set of pre-images of $x$ is invariant by $\sigma^p$. Then, for each $m\in\NN_0$, the element $\hat{a}_{m+p}$ in $I_r$, whose terms are $\alpha_{m+p}^1,\ldots,\alpha_{m+p}^r$, is the same as the element $\hat{a}_m\in I_r$, because its entries, $\alpha_{m}^1,\ldots,\alpha_{m}^r$, are the same (although not necessarily in the same order). Therefore we conclude that $\hat{a}_{m+p}=\hat{a}_m$, that is to say, $\underline{\hat{a}}\in Per_p(\sigma_r^+)$.\\

The next Proposition will give a formula for the number of periodic points of $f$. First notice that, if $I_r\neq\emptyset$, then $I_{r'}\neq\emptyset$ for $r'<r$.

\begin{prop}\label{M22}
For all $p\in\NN$,
\[
N_p(f)=\sum_{r=1}^L(-1)^{r-1}\tr((B^{(r)})^p)
\]
where $L$ is the largest value of $r$ for which $I_r\neq\emptyset$.
\end{prop}

\begin{proof}
Given $x\in Per_p(f)$, consider the function given by
\[
\Phi(x)=\sum_{t=1}^L\left(\sum_{\underline{\hat{a}}\,\in\,\hat{\Pi}_t^{-1}(x)\bigcap Per_p(\sigma_r^+)} (-1)^{t-1} sgn(\nu)\right)
\]
where $\nu$ is the unique permutation in $S_t$ such that $\alpha_p^{\nu(i)}=\alpha_0^{i},\forall i\in[t]$, being $\underline{\alpha}^i$, for $i\in[t]$, the elements of $\Sigma_A^+$ constructed as before.

We want to show that $\Phi(x)=1$. Let $\Pi^{-1}(x)=\left\{\underline{\alpha}^1,\ldots,\underline{\alpha}^r\right\}$ and $\mu$ be the permutation such that $\sigma^p(\underline{\alpha}^i)=\underline{\alpha}^{\mu(i)},\forall i\in\left[r\right]$, that is to say, the permutation induced by the action of $\sigma^p$ on $\Pi^{-1}(x)$. We can write $\mu$ as the product of disjoint cycles $\mu_1,\ldots,\mu_s$ (eventually with length 1) which act on the sets $K_1,\ldots,K_s$, respectively, and these sets form a partition of [r].

Given $\underline{\hat{a}}\in\hat{\Pi}_t^{-1}(x)$, we can build $t$ distinct pre-images of $x$ under $\Pi$, with $t\leq r$. Let $J\subseteq [r]$ be such that these pre-images are $(\underline{\alpha}^j)_{j\in J}$. If we suppose additionally that $\underline{\hat{a}}\in Per_p(\sigma_r^+)$, then $J$ is invariant under $\nu$, so we can write $J=\bigcup_{m\in B}K_m$ for some $\emptyset\neq B\subseteq [s]$. On the other hand, for each nonempty subset $B$ of $[s]$, we can take $J=\bigcup_{m\in B}K_m$ and associate to it a sequence $\underline{\hat{a}}$ given by the set of distinct pre-images $(\underline{\alpha}^j)_{j\in J}$. So, for each $t\in[L]$ and $\underline{\hat{a}}\,\in\,\hat{\Pi}_t^{-1}(x)\bigcap Per_p(\sigma_r^+)$, we can associate a unique nonempty subset $B$ of $[s]$, and we have
\[
t=card(J)=card\left(\bigcup_{m\in B}K_m\right)=\sum_{m\in B}card(K_m).
\]
Since $\mu_m$ is a cycle of length $card(K_m)$, we have
\[
sgn(\nu)=\prod_{m\in B}sgn(\mu_m)=\prod_{m\in B}(-1)^{card(K_m)+1}=(-1)^{t+card(B)}.
\]
Hence,
\[
(-1)^{t-1} sgn(\nu)=(-1)^{2t-1+card(B)}=-(-1)^{card(B)}
\]
and
\begin{eqnarray*}
\Phi(x)&=&\sum_{t=1}^L\left(\sum_{\underline{\hat{a}}\,\in\,\hat{\Pi}_t^{-1}(x)\bigcap Per_p(\sigma_r^+)} (-1)^{t-1} sgn(\nu)\right)\\
&=& -\sum_{\emptyset\,\neq \, B\subseteq[s]}(-1)^{card(B)}=-\sum_{q=1}^s\sum_{B\,\subseteq\,[s],\,card(B)=q}(-1)^{card(B)} \\
&=& -\sum_{q=1}^s {s \choose q}(-1)^q={s \choose 0}(-1)^0-\sum_{q=0}^s {s \choose q}(-1)^q \\
&=& 1-(1-1)^s=1.
\end{eqnarray*}
The inclusion $Per_p(\sigma_r^+)\subseteq \hat{\Pi}_t^{-1}(Per_p(f))$ now yields
\begin{eqnarray*}
N_p(f) &=& \sum_{x\in Per_p(f)}\Phi(x) \\
&=& \sum_{x\,\in \, Per_p(f)}\sum_{t=1}^L\left( \sum_{\underline{\hat{a}}\,\in\,\hat{\Pi}_t^{-1}(x)\bigcap Per_p(\sigma_r^+)}(-1)^{t-1} sgn(\nu)\right) \\
&=& \sum_{t=1}^L\left(\sum_{\underline{\hat{a}}\,\in \, Per_p(\sigma_r^+)}(-1)^{t-1} sgn(\nu)\right) \\
&=& \sum_{t=1}^L (-1)^{t-1}\left(\sum_{\underline{\hat{a}}\, \in \, Per_p(\sigma_r^+)} sgn(\nu)\right).
\end{eqnarray*}
Concerning the last summand, let $(\hat{a}_0,...,\hat{a}_n)$ be an admissible sequence of length $n+1$ for the matrix $A^{(t)}$ and let $\mu_m$ be the permutation which ensures that $A^{(t)}_{\hat{a}_m \hat{a}_{m+1}}=1$, for $m\in\left\{0,1,...,n-1\right\}$. Then we have $B^{(t)}_{\hat{a}_m \hat{a}_{m+1}}=sgn(\mu_m)$. Consider the permutations $\nu_m$ given by
\begin{center}
$\nu_0=id$

$\nu_m=\mu_{m-1}\circ...\circ\mu_0$
\end{center}
so that $\nu_{m+1}=\mu_m\circ\nu_m$ for $m\in\left\{0,1,...,n-1\right\}$. Let $S(\hat{a}_0,\hat{a}_n,n)$ denote the set of admissible sequences of length $n+1$ which start at $\hat{a}_0$ and end at $\hat{a}_n$. Then we can show by induction over $n$ that
\[
\sum_{S(\hat{a}_0,\hat{a}_n,n)} sgn(\nu_n)=((B^{(t)})^n)_{\hat{a}_0 \hat{a}_n}.
\]
For $n=1$, given two elements $\hat{a}_0,\hat{a}_1\in I_t$ we have $\nu_1=\mu_0$, so
\[
sgn(\nu_1)=sgn(\mu_0)=(B^{(t)})_{\hat{a}_0 \hat{a}_1}.
\]
Suppose the assertion is true for $n=m-1$. Then, for $n=m$,
\begin{eqnarray*}
\sum_{S(\hat{a}_0,\hat{a}_m,m)} sgn(\nu_m) &=& \sum_{S(\hat{a}_0,\hat{a}_m,m)} sgn(\mu_{m-1})sgn(\nu_{m-1}) \\
&=& \sum_{\left\{\hat{a}_{m-1}\in I_r:A^{(t)}_{\hat{a}_{m-1}\hat{a}_m}=1\right\}}\left(\sum_{S(\hat{a}_0,\hat{a}_{m-1},m-1)} sgn(\nu_{m-1})\right)sgn(\mu_{m-1}) \\
&=& \sum_{\left\{\hat{a}_{m-1}\in I_r:A^{(t)}_{\hat{a}_{m-1}\hat{a}_m}=1\right\}}((B^{(t)})^{m-1})_{\hat{a}_0 \hat{a}_{m-1}}B^{(t)}_{\hat{a}_{m-1} \hat{a}_m} \\
&=& ((B^{(t)})^m)_{\hat{a}_0 \hat{a}_m}.
\end{eqnarray*}
In particular,
\[
\sum_{S(\hat{a}_0,\hat{a}_0,n)} sgn(\nu_n)=((B^{(t)})^n)_{\hat{a}_0 \hat{a}_0}
\]
For each sequence $\underline{\hat{a}}\in Per_p(\sigma_r^+)$ we can associate a unique element of $S(\hat{a}_0,\hat{a}_0,p)$ which verifies $\nu_p=\nu$. So
\[
\sum_{\underline{\hat{a}}\in Per_p(\sigma_r^+)} sgn(\nu)=
\sum_{\hat{a}_0\in I_t}((B^{(t)})^p)_{\hat{a}_0 \hat{a}_0}=
\tr((B^{(t)})^p).
\]
Then we finally conclude that
\[
N_p(f)=\sum_{t=1}^L (-1)^{t-1}\tr((B^{(t)})^p).
\]
\end{proof}

\begin{theo}\label{M23}
The $\zeta$-function of $f$ is rational.
\end{theo}

\begin{proof}
We have seen that, for any $n$,
\[
N_n(f)=\sum_{r=1}^L(-1)^{r-1}\tr((B^{(r)})^n)
=\sum_{r\in[L],\,r\,odd}\tr((B^{(r)})^n)-\sum_{r\in[L],\,r\,even}\tr((B^{(r)})^n).
\]
So
\begin{eqnarray*}
\zeta_f(z) &=& \exp\left(\sum_{n=1}^\infty\frac{\sum_{r\in[L],\,r\,odd}\tr((B^{(r)})^n)-\sum_{r\in[L],\,r\,even}\tr((B^{(r)})^n)}{n}z^n\right) \\
&=& \frac{\exp\left(\sum_{n=1}^\infty\frac{\sum_{r\in[L],\,r\,odd}\tr((B^{(r)})^n)}{n}z^n\right)}{\exp\left(\sum_{n=1}^\infty\frac{\sum_{r\in[L],\,r\,even}\tr((B^{(r)})^n)}{n}z^n\right)}\\
&=& \frac{\prod_{r\in[L],\,r\,odd}\exp\left(\sum_{n=1}^\infty\frac{\tr((B^{(r)})^n)}{n}z^n\right)}{\prod_{r\in[L],\,r\,even}\exp\left(\sum_{n=1}^\infty\frac{\tr((B^{(r)})^n)}{n}z^n\right)} \\
&=& \frac{\prod_{r\in[L],\,r\,odd}\frac{1}{\det(I-zB^{(r)})}}{\prod_{r\in[L],\,r\,even}\frac{1}{\det(I-zB^{(r)})}} \\
&=& \frac{\prod_{r\in[L],\,r\,even}\det(I-zB^{(r)})}{\prod_{r\in[L],\,r\,odd}\det(I-zB^{(r)})}
\end{eqnarray*}
which is clearly a rational function. It is also interesting to notice that the zeta function's coefficients are integer numbers.
\end{proof}

\noindent \textbf{Question}: When $f$ is a subshift of finite type associated to an irreducible matrix $A$, then $A^{1}=A$ and the radius of convergence of $\zeta_f$ is equal to $\log\lambda$, where $\lambda$ is the simple eigenvalue given by Perron-Froebenius' Theorem. What may be said in the general case? Do the matrices $A^{r}$ and $B^{r}$ yield some information of the same kind?

\section{Proof of Theorem~\ref{MainTheorem2}}\label{Entropy}

We are now assuming that $f: K \rightarrow K$ has the properties assigned to one basic component $\Lambda_i^{(m)}$, that is,
\begin{itemize}
\item[(C1)] $f(K)=K$.
\item[(C3)] $K=\overline{Per(f)}$.
\item[(C4)] $f$ is Ruelle-expanding.
\item[(C5)] For any open nonempty subset $V$ of $K$ there is $N\in \mathbb{N}$ such that $f^N(V)=K$.
\end{itemize}

From Corollary~\ref{PerEntr}, one already knows that $\wp(f)\leq h(f)$. To get the other inequality, it is enough to prove the following estimate.

\begin{prop}\label{PERSEP}
Let $\vep$ be a constant of expansivity of $f$ and $\vep_0<\vep/4$. Then there exists a constant $C>0$ and a positive integer $n_0$ such that, for all $n \geq n_0$, we have $N_n(f)\geq C\, s_n(\vep_0, K)$.
\end{prop}

\begin{proof}

\begin{lem}\label{mixing}
Given $\delta>0$ there is $N_\delta \in \mathbb{N}$ such that, for all $m \geq N_\delta$ and any $x \in K$, we have $f^m(B_\delta(x))=K$.
\end{lem}

\begin{proof}
As $K$ is compact, we may choose a finite set of points $p_1, p_2,\ldots,p_\ell$ such that every $x \in K$ is within a distance smaller than $\displaystyle \frac{\delta}{2}$ to some $p_j$. By condition (C5), there are positive integers $N_1, N_2, \ldots, N_\ell$ such that $f^{m}(B_{\frac{\delta}{2}}(p_i))=K$, for any $1\leq i\leq \ell$ and all $m \geq N_i$. Take $N_\delta =\max \,\{N_1,\ldots,N_\ell\}$. Then, as $B_\delta(x)\supseteq B_{\frac{\delta}{2}}(p_j)$, we have $f^{N_\delta}(B_\delta(x))\supseteq f^{N_\delta}(B_{\frac{\delta}{2}}(p_j))\supseteq K.$ Thus the same holds for all $m \geq N_\delta$.
\end{proof}

Consider any $0< \tau <\displaystyle \frac{\vep_0}{8}$ and take $\delta =\frac{1}{2} \,\min \,\{\tau, \alpha_\tau\}$ (the value $\alpha$ is obtained in Corollary~\ref{PER}). Fix $x \in K$ and the dynamical ball
$$B(n-N_\delta, \delta, x)=\displaystyle \left\{y\in K : d(f^j(x),f^j(y))< \delta,\,\,\forall j\in\{0,\ldots,n-N_\delta\}\right\}.$$

\begin{lem}\label{EXISTPER}
$Per_n(f)\, \cap \, B(n-N_\delta, 2\tau, x) \neq \emptyset$ for all $n \geq N_\delta+1$.
\end{lem}

\begin{proof}
Take a contractive branch $g:B_\delta(f^{n-N_\delta}(x)) \longrightarrow K$ of $f^{n-N_\delta}$ such that $g(f^{n-N_\delta}(x))=x$. By Lemma~\ref{mixing} we know that $f^{N_\delta} \left(B_\delta(f^{n-N_\delta}(x))\right)=K$, and so, as $(f^{n-N_\delta} \circ g)(y)=y$ for all $y \in B_\delta(f^{n-N_\delta}(x))$, we get
$$f^n(g(B_\delta(f^{n-N_\delta}(x))))=f^{N_\delta}(f^{n-N_\delta}(g(B_\delta(f^{n-N_\delta}(x)))))=f^{N_\delta}(B_\delta(f^{n-N_\delta}(x)))=K.$$
Moreover, by Proposition~\ref{M8}, $g(B_\delta(f^{n-N_\delta}(x)))=B(n-N_\delta, \delta, x)$, and so we may find $z \in B(n-N_\delta, \delta, x)$ such that $f^n(z) \in B(n-N_\delta, \delta, x)$. As $\delta < \alpha_\tau$, Corollary~\ref{PER} yields a point $w$ such that $f^{n}(w)=w$ and $d(f^j(w),f^j(z))<\tau$ for all $0\leq j \leq n$. Therefore, for $0\leq j \leq n-N_\delta$, we have
$$d(f^j(x),f^j(w))\leq d(f^j(x),f^j(z))+ d(f^j(z),f^j(w)) < \delta + \tau < 2\tau.$$
\end{proof}

\begin{cor}\label{Final}
$N_n(f)\geq s_{n-N_\delta}(4\tau, K)$ for all $n \geq N_\delta+1$.
\end{cor}

\begin{lem}\label{sum}
Fix two positive integers $n_1, n_2$ and $\gamma >0$. Then
$$\displaystyle s_{n_1+n_2}\,(\gamma, K)\leq s_{n_1}\left(\frac{\gamma}{2},K\right) s_{n_2}\left(\frac{\gamma}{2},K\right).$$
\end{lem}

\begin{proof}
Suppose that $E\subseteq K$ is such that, for any $x, y \in E$, $x\neq y$, there is $t \in \,[0, n_1+n_2[$ for which $d(f^t(x),f^t(y))> \gamma$. Take $S_1\subseteq K$ a maximal $(n_1,\frac{\gamma}{2})$-separated set and $S_2 \subseteq K$ such that $f^{n_1}(S_2)$ is a maximal $(n_2,\frac{\gamma}{2})$-separated set. To construct $S_2$, consider a maxi\-mal $(n_2,\frac{\gamma}{2})$-separated set $T_2=\{d_1, \ldots, d_M\}$ and define $S_2=\{c_1, \ldots, c_M\}$ such that $c_j\in f^{-n_1}(\{d_j\})$, for each $j \in \{1, \ldots, M\}$, which is possible since $f(K)=K$. \\

For each $i \in \{1,2\}$, define the maps $\psi_i: E \rightarrow S_i$ by the conditions:
\begin{itemize}
\item For any $x \in E$ and all $t \in \,[0,n_1[$, $d(f^t(x),f^t(\psi_1(x)))\leq \frac{\gamma}{2}$.
\item For any $x \in E$ and all $t \in \,[n_1,n_1+n_2[$, $d(f^t(x),f^t(\psi_2(x)))\leq \frac{\gamma}{2}$.
\end{itemize}
Such a $\psi_i(x)$ may be found in $S_i$, otherwise either $x$ would not belong to $S_1$, and so the set $S_1 \, \cup \, \{x\}$ would be $(n_1,\frac{\gamma}{2})$-separated, contradicting the maximality of $S_1$; or $x$ would not belong to $S_2$, and therefore the set $f^{n_1}(S_2 \, \cup \, \{x\})$ would be $(n_2,\frac{\gamma}{2})$-separated, contradicting the maximality of $f^{n_1}(S_2)$. Moreover, the map
\begin{eqnarray*}
\psi:& E & \rightarrow S_1 \times S_2 \\
& x & \mapsto (\psi_1(x), \psi_2(x))
\end{eqnarray*}

\noindent is injective because, given $x, y \in E$ with $\psi(x)=\psi(y)$, then, for all $t \in \,[0,n_1[$,
\begin{eqnarray*}
d(f^t(x), f^t(y)&\leq& d(f^t(x),f^t(\psi_1(x)) + d(f^t(\psi_1(x), f^t(y)) \\
&=&d(f^t(x),f^t(\psi_1(x)) + d(f^t(\psi_1(y), f^t(y)) \leq \gamma
\end{eqnarray*}
and, for all $t \in \,[n_1, n_1 + n_2[$,
$$d(f^t(x), f^t(y)\leq d(f^t(x),f^t(\psi_2(x)) + d(f^t(\psi_2(x), f^t(y)) \leq \gamma$$
which contradicts the definition of $E$ if $x \neq y$.
\end{proof}

\begin{cor}
$N_n(f)\geq \,\frac{s_{n}(8\tau, K)}{s_{N_\delta}(4\tau, K)}$ for all $n \geq N_\delta+1$
\end{cor}

\begin{proof} Starting with Corollary~\ref{Final} and applying Lemma~\ref{sum} to $n_1=n-N_\delta$ and $n_2=N_\delta$, we get
$$N_n(f)\geq s_{n-N_\delta}(4\tau, K) \geq \frac{s_{n}(8\tau, K)}{s_{N_\delta}(4\tau, K)}.$$
\end{proof}

Let $c$ denote the factor $\frac{1}{s_{N_\delta}(4\tau, K)}$. As $s_m(\vep, K) \leq s_m(\gamma, K)$, for all $m$ and all $\gamma < \vep$, and $8\tau < \vep_0$, we finally reason that, for $n \geq N_\delta+1$,
$$N_n(f)\geq c\,{s_{n}(\vep_0, K)}.$$
\end{proof}

From Propositions~\ref{M10} and \ref{PERSEP}, we deduce that
\begin{eqnarray*}
\wp(f)&=&\limsup_{n\rightarrow\infty} \, \frac{1}{n}\log(\max\{N_n(f),1\}) \\
&\geq & \limsup_{n \rightarrow + \infty} \, \frac{1}{n}\, \left[\log (c) + \log s_n(\vep_0, K)\right] \\
&=&\lim_{n \rightarrow + \infty} \, \frac{1}{n} \, \log s_n(\vep_0, K) = h(f).
\end{eqnarray*}
Thus $\wp(f)=h(f)$.

\ms

Moreover, if $\mathcal{B}$ and $\mathcal{C}$ are the covers of $K$ by open balls of radius $2\vep_0$ and $\frac{\vep_0}{2}$, respectively, then, from Propositions \ref{TOPENTR1and2} and \ref{PERSEP}, we get, for all $n\geq N_\delta+1$,
$$c\,H(\bigvee_{i=0}^{n-1}\,f^{-i}\, \mathcal{B}) \leq  c\,s_n(\vep_0, K) \leq N_n(f) \leq s_n(\vep_0, K) \leq H(\bigvee_{i=0}^{n-1}\,f^{-i}\, \mathcal{C})$$
and so, as, by Proposition~\ref{M10},
$$h(f)= h(f, \mathcal{B})=h(f, \mathcal{C}),$$
the limit
$$\lim_{n \rightarrow + \infty}\,\frac{1}{n}\,\,\log N_n(f)$$
exists and is equal to $h(f)$.
$\square$


\bigskip

\section{Bibliography}

\begin{thebibliography}{00}

\bibitem{AM} M. Artin, B. Mazur, \emph{On periodic points}, Annals of Mathematics 81, 82-99, 1965.

\bibitem{BL} R. Bowen, O. E. Lanford III, \emph{Zeta functions of restrictions of the shift transformation}, Proceedings of Symposia in Pure Mathematics, Vol XIV, 43-49, American Mathematical Society, 1970.

\bibitem{Bo2} R. Bowen, \emph{Periodic points and measures for Axiom A diffeomorphisms}, Transactions of the American Mathematical Society, Vol 154, 377-397, 1971.

\bibitem{Bo} R. Bowen, \emph{Equilibrium states and the ergodic theory of Anosov diffeomorphisms}, Lecture Notes in Mathematics, 470, Springer, 1975.

\bibitem{BRW} M. Baake, J. Roberts, A. Weiss, \emph{Periodic orbits of linear endomorphisms on the 2-torus and its lattices}, Nonlinearity 21, 2427-2446, 2008.

\bibitem{Cra} M. Craizer, \emph{Teoria erg\'odica das transforma\c c\~oes expansoras}, Informes de Matem\'atica, S\'erie E-018-Agosto/85, IMPA (Brasil).

\bibitem{ES} D. Epstein, M. Shub, \emph{Expanding endomorphisms of flat manifolds}, Topology 7, 139-141, 1968.



\bibitem{H2} F. Hofbauer, \emph{On intrinsic ergodicity of piecewise monotonic transformations with positive entropy II}, Israel Journal of Mathematics 38, 107-115, 1981.

\bibitem{JL} B. Jiang, J. Llibre, \emph{Minimal sets of periods for torus maps}, Discrete and Continuous Dynamical Systems 4, 301-320, 1998.

\bibitem{M} A. Manning, \emph{Axiom A diffeomorphisms have rational zeta functions}, Bulletin of the London Mathematical Society 3, 215, 1971.

\bibitem{MT} J. Milnor, W. Thurston, \emph{On iterated maps of the interval}, Lecture Notes in Mathematics 1342, Springer, 465-563, 1988.

\bibitem{MS} M. Misiurewicz, W. Szlenk, \emph{Entropy of piecewise monotone mappings}, Studia Mathematica LXVII, 45--63, 1980.

\bibitem{N} H. E. Nusse, \emph{Chaotic maps with rational zeta function}, Transactions of the American Mathematical Society 304, 2, 705-719, 1987.

\bibitem{Ru1} D. Ruelle, \emph{Statistical mechanics of a one dimensional lattice gas}, Communications in Mathematical Physics, 267-278, 1968.

\bibitem{Ru3} D. Ruelle, \emph{Zeta-functions for expanding maps and Anosov flows}, Invent. Math. 34, N.3, 231–242, 1976.

\bibitem{Ru2} D. Ruelle, \emph{Thermodynamic formalism}, Addison-Wesley, 1978.


\bibitem{Shu3} M. Shub, \emph{Endomorphisms of compact differentiable manifolds}, American Journal of Mathematics 91, 129-155, 1969.

\bibitem{Shu} M. Shub, \emph{Global stability of dynamical systems}, Springer, 1987.


\bibitem{Tau} R. Tauraso, \emph{Sets of periods for expanding maps on flat manifolds}, Monatshefte f\"{u}r Mathematik 128, n. 2, 151-157, 1999.

\bibitem{Wa} P. Walters, \emph{An introduction to ergodic theory}, Springer, 1975.

\bibitem{Wi} R. F. Williams, \emph{Zeta function in global analysis}, Proceedings of Symposia in Pure Mathematics, Volume XIV, 335-339, American Mathematical Society, 1970.

\end{thebibliography}
\end{document}